\documentclass[final]{amsart}

\usepackage[foot]{amsaddr}
\usepackage{booktabs}
\usepackage{fullpage}
\usepackage{tristan}
\usepackage{algorithm}
\usepackage{algpseudocode}

\usepackage{pgfplots}
\usepackage{pgfplotstable}
\usepackage{subcaption}
\usepackage{tikz}
\usepackage[giveninits=true,eprint=false,url=false,style=alphabetic]{biblatex}
\definecolor{seaborngreen}{rgb}{0.3333333333333333, 0.6588235294117647, 0.40784313725490196}

\title{Geometry, Energy and Sensitivity in Stochastic Proton Dynamics}

\author[1]{Veronika Chronholm}
\thanks{Corresponding author: Veronika Chronholm}
\email[1]{mvc34@bath.ac.uk}

\author[2]{Tristan Pryer}

\address[1,2]{Institute for Mathematical Innovation\\ University of
  Bath, Bath, UK.}
\address[1,2]{Department of Mathematical Sciences\\ University of Bath, Bath, UK.}

\begin{document}

\begin{abstract}
  We develop numerical schemes and sensitivity methods for stochastic
  models of proton transport that couple energy loss, range straggling
  and angular diffusion. For the energy equation we introduce a
  logarithmic Milstein scheme that guarantees positivity and achieves
  strong order one convergence. For the angular dynamics we construct
  a Lie-group integrator. The combined method maintains the natural
  geometric invariants of the system.

  We formulate dose deposition as a regularised path-dependent
  functional, obtaining a pathwise sensitivity estimator that is
  consistent and implementable. Numerical experiments confirm that the
  proposed schemes achieve the expected convergence rates and provide
  stable estimates of dose sensitivities.
\end{abstract}

\maketitle


\section{Introduction}

Stochastic models of proton transport play an important role in
high-precision applications such as proton therapy, where accurate
characterisation of energy loss and angular scattering is essential to
treatment planning and dose delivery. Deterministic formulations,
based on the continuous slowing down approximation or transport
equations, capture average behaviour and admit analytic
approximations, but they neglect statistical fluctuations
\cite{burlacu2023deterministic,ashby2025positivity}. Stochastic formulations, by contrast,
describe individual particle paths, accounting for both deterministic
energy loss due to inelastic interactions and random angular
deflections from Coulomb scatter, with additional variability
introduced by range straggling
\cite{crossley2024jump,kyprianou2025unified}. Analytic models of
either type provide reduced-order descriptions of dose, while
numerical approaches range from PDE solvers and pencil-beam algorithms
to Monte Carlo and SDE-based simulations
\cite{ashby2025efficient}. For physical observables such as energy
deposition, fluence and dose \cite{NewhauserZhang:2015}, it is the
stochastic, numerically approximated models that are required to
capture fluctuations and uncertainty at clinically relevant accuracy.

In radiotherapy and related fields, reliable computation of dose and
fluence demands numerical approximations that faithfully capture the
geometry of particle transport and the stochastic nature of energy
loss. For protons, this means maintaining positivity of the energy,
ensuring that directional vectors remain on the unit sphere and
correctly representing the randomness inherent in multiple scattering
and range straggling. Conventional numerical schemes, even in the
deterministic case, often violate one or more of these requirements
\cite{hairer2006geometric}. The energy variable may become negative
under naive discretisation, or angular updates may drift off the
sphere.  Moreover, many practical tasks, such as machine-learning
guided treatment optimisation, uncertainty quantification, or
calibration to experimental data, depend not just on forward
simulations but also on sensitivities of dose and fluence with respect
to model parameters (e.g. stopping-power coefficients or straggling
parameters).  Finite-difference estimates of these sensitivities are
routinely employed in clinical settings, yet they tend to be biased
and have large variance when applied to stochastic systems, making
them computationally inefficient for high-accuracy work. There is a
clear need for numerical methods that simultaneously respect the
physical constraints of proton transport and provide efficient
pathwise sensitivity information with low empirical variance in a
unified framework.

In this work we combine a structure-preserving discretisation for the
simulation and sensitivity analysis of a stochastic proton transport
model with continuous energy loss and angular diffusion. Our main
contributions are:
\begin{enumerate}
\item We reformulate the dynamics using a logarithmic transformation
  of the energy and an exponential map for angular updates. This
  guarantees exact positivity of energy and preservation of the
  unit-norm constraint on direction at the discrete level. On this
  basis we construct (i) a geometric Euler scheme of strong order
  $1/2$ and (ii) a higher-order scheme that couples a Milstein
  discretisation of the log-energy SDE with a geometric integrator for
  angular evolution, achieving strong order $1$.
\item We derive and discretise coupled forward sensitivity equations
  for the stopping-power and straggling parameters. These
  sensitivities satisfy linear SDEs driven by the same Brownian motion
  as the state, so that the coupled state-sensitivity scheme achieves
  strong order~$1$ convergence. This yields pathwise gradient
  estimators for regularised dose-type observables, and in our
  numerical experiments these exhibit substantially lower variance
  than finite-difference benchmarks.
\item We present numerical experiments that demonstrate the efficacy
  of the proposed methods. The geometric discretisations preserve key
  physical features of proton transport, including Bragg peaks and
  angular spreading, while the pathwise sensitivities converge with
  lower variance than finite-difference benchmarks. We further show
  how incorporating energy-dependent angular diffusion (Moli\`ere
  theory) modifies dose distributions and sensitivities, highlighting
  the flexibility of the framework.
\end{enumerate}

Numerical simulation of stochastic differential equations is a well
developed area, with classical treatments of strong approximation
theory and positivity-preserving schemes appearing in
\cite{kloeden1992stochastic,milstein2004stochastic}. Recent work has
extended this theory to include tamed and truncated schemes for SDEs
with super-linear coefficients and positivity constraints
\cite{deng2024positivity,sabanis2013note}. For manifold-valued
stochastic dynamics, geometric and intrinsic approaches have received
substantial attention, including Lie-group methods
\cite{iserles2000lie}, projected and constrained dynamics in molecular
simulation \cite{stoltz2010free}, Langevin-type sampling and
estimation on manifolds \cite{bharath2025sampling}, frozen-flow and
higher-order integration methods on Riemannian manifolds
\cite{bronasco2025high}, and recent symplectic techniques on reductive
Lie groups \cite{luesink2026symplectic}. Questions of long-time
accuracy and invariant-measure approximation have also been studied
extensively in the stochastic geometric integration literature
\cite{abdulle2014high,laurent2022order}. Structure-preserving
stochastic integrators for manifold-valued systems have moreover been
developed in the context of matrix Lie groups, with
Runge--Kutta--Munthe--Kaas methods and related constructions extended
to SDEs in \cite{muniz2022stochastic}. The approach adopted here draws
on these developments, combining a logarithmic transformation for the
energy variable with manifold-aware angular updates to preserve
positivity and the spherical constraint in a stochastic proton
transport setting.

In proton therapy, dose modelling is typically performed using
deterministic or Monte Carlo methods \cite{bortfeld1997analytical,
  paganetti2018proton, faddegon2020topas,kusch2023kit}, with recent
machine learning and dynamical low-rank solvers for proton transport
\cite{stammer2024deterministic,stammer2025high}.  These approaches
generally decouple model simulation from parameter sensitivity
analysis, or rely on finite difference estimators that suffer from the
tradeoff between bias and high variance, as well as numerical
instability.  Recent work by \cite{ashby2025efficient} applied
sensitivity analysis to an analytic model of proton transport, while a
more complex stochastic model incorporating jump processes was
introduced in \cite{crossley2024jump}. An MCMC-based framework for
inverse inference of dose delivery from experimental data has also
been proposed in \cite{CoxHattamKyprianouPryer:2023}. Pathwise
sensitivity methods are well established in the context of financial
mathematics (see e.g.\ \cite{giles2006smoking, broadie1996estimating,
  capriotti202315}, or Chapter~7 of \cite{glasserman2004monte}) and
stochastic simulation.  Their application in radiation transport has
been limited, though early examples include physical simulation
contexts such as chemical kinetics and photon propagation
\cite{sheppard2012pathwise}. Other recent work includes \cite{caflisch2024adjoint},
where these methods are discussed in the wider context of methods 
for Monte Carlo gradient estimation, with application to photon propagation
as well as the nonlinear Boltzmann equation.
The present work adapts these techniques
to a geometric SDE model of proton motion, coupling
structure-preserving integrators with forward sensitivity equations to
enable efficient pathwise gradient computation for physically
meaningful observables such as dose.  Connections to Feynman-Kac
representations for energy deposition problems are discussed in light
of recent work on stochastic particle transport
\cite{yang2021feynman,kyprianou2025unified}. 

The remainder of the paper is structured as follows.  In
Section~\ref{sec:physical-sde}, we describe the physical setting and
introduce the stochastic differential equation (SDE) model for proton
transport, incorporating continuous slowing down and angular
diffusion. We identify the key structural properties of the model,
including energy positivity, angular manifold constraints and
long-time behaviour. In Section~\ref{sec:euler}, we examine the
standard Euler-Maruyama discretisation and highlight its failure to
preserve the structure of the continuous system. This motivates a
geometric update where the angular component is corrected using the
exponential map. Section~\ref{sec:high-order} presents higher-order
discretisation for both energy and direction, including a Milstein
scheme for the log-transformed energy and a second-order RKMK-like
integrator for angle. Section \ref{sec:dose-sensitivity} describes the
dose functional and how its pathwise sensitivity can be
computationally evaluated. Section~\ref{sec:sensitivity} then develops
the pathwise sensitivity equations, describes their numerical
discretisation and applies them to the estimation of dose
sensitivities with respect to physical parameters.  Section
\ref{sec:numerics} summarises extensive numerical experiments and we
conclude in Section \ref{sec:conclusion}.

\section{Physical Model and Stochastic Formulation}
\label{sec:physical-sde}

The motion of a proton undergoing interactions in a medium is governed
by a combination of deterministic energy loss and stochastic angular
deflections. The proton moves along a trajectory described by its
position $X_t \in \mathbb{R}^3$, direction $\Omega_t \in
\mathbb{S}^2$, and energy $E_t > 0$. In the clinical energy range
(50–150 MeV), the dominant interactions influencing proton transport
are illustrated in Figure~\ref{fig:atom}. These include inelastic
collisions with electrons, elastic Coulomb scattering with nuclei,
and, less frequently, inelastic nuclear reactions.
\begin{figure}[h!]
    \centering
    \renewcommand{\proton}[1]{%
    \shade[ball color=red!80!white, draw=black, line width=0.5pt] (#1) circle (.25);
    \draw[black] (#1) node{$+$};
}

\renewcommand{\neutron}[1]{%
    \shade[ball color=lime!70!green, draw=black, line width=0.5pt] (#1) circle (.25);
}

\renewcommand{\electron}[3]{%
    \draw[rotate = #3, color=gray!60!white, line width=0.8pt] (0,0) ellipse (#1 and #2);
    \shade[ball color=yellow!90!white, draw=black, line width=0.5pt] (0,#2)[rotate=#3] circle (.1);
}

\newcommand{\legendelectron}[1]{%
    \shade[ball color=yellow!90!white, draw=black, line width=0.5pt] (#1) circle (.1);
}

\renewcommand{\nucleus}{%
    \neutron{0.1,0.3}
    \proton{0,0}
    \neutron{0.3,0.2}
    \proton{-0.2,0.1}
    \neutron{-0.1,0.3}
    \proton{0.2,-0.15}
    \neutron{-0.05,-0.12}
    \proton{0.17,0.21}
}

\renewcommand{\inelastic}[2]{
  \proton{#1,#2};
  \draw[->,thick,cyan!80!white](#1+0.5,#2)--(4,-3.7); 
  \draw[->,thick,cyan!80!white](0,-3)--(-0.3,-4); 
  \shade[ball color=yellow] (-0.3,-4) circle (.1); 
}

\renewcommand{\elastic}[2]{
  \proton{#1,#2};
  \draw[->,thick,orange,bend right=90](#1+0.5,#2) to  [out=-30, in=-150] (4,3.);
}

\renewcommand{\protoncollision}[3]{
  \proton{#1,#2};
  \draw[->,thick,red](#1+0.5,#2)--(-0.5,0);%
  \draw[snake=coil, line after snake=0pt, segment aspect=0,%
    segment length=5pt,color=red!80!blue] (0,0)-- +(4,2)%
  node[fill=white!70!yellow,draw=red!50!white, below=.01cm,pos=1.]%
  {$\gamma$};%
  \draw[->,thick,red](#1+0.5,#2)--(-0.5,0);%
  \draw[->,thick,red](0.5,0)--(3.7,-1.8);%
  \neutron{4,-2};  
}

\begin{tikzpicture}[scale=0.5]
    \nucleus
    \electron{1.5}{0.75}{80}
    \electron{1.2}{1.4}{260}
    \electron{4}{2}{30}
    \electron{4}{3}{180}
    \protoncollision{-6.}{0.}{160}
    \inelastic{-6.}{-2.}
    \elastic{-6.}{2.}

      \begin{scope}[shift={(8,-1)}]        
        \draw [thick,rounded corners=2pt] (0,-4) rectangle (8,4); 
        \node at (4, 3.5) {\textbf{Legend}};        
        \proton{0.5, 3}
        \node[anchor=west, font=\footnotesize] at (1.5,3) {Proton}; 
        \neutron{0.5, 2}
        \node[anchor=west, font=\footnotesize] at (1.5,2) {Neutron};         
        \legendelectron{0.5, 1}
        \node[anchor=west, font=\footnotesize] at (1.5,1) {Electron};        
        \draw[->,thick,red] (0.5,0) -- +(1,0);
        \node[anchor=west, font=\footnotesize] at (1.5,0) {Nonelastic collision};        
        \draw[->,thick,cyan!80!white] (0.5,-1) -- +(1,0);
        \node[anchor=west, font=\footnotesize] at (1.5,-1) {Inelastic interaction};         
        \draw[->,thick,orange] (0.5,-2) -- +(1,0);
        \node[anchor=west, font=\footnotesize] at (1.5,-2) {Elastic interaction}; 
        \draw[snake=coil, line after snake=0pt, segment aspect=0,
          segment length=5pt,color=red!80!blue] (0.5,-3) -- +(1,0);
        \node[anchor=west, font=\footnotesize] at (1.5,-3) {prompt-$\gamma$ emission};
      \end{scope}
\end{tikzpicture}
    \caption{\em The three main interactions of a proton with matter.
      A \textcolor{red}{nonelastic} proton--nucleus collision, an
      \textcolor{cyan!80!white}{inelastic} Coulomb interaction with
      atomic electrons and \textcolor{orange}{elastic} Coulomb
      scattering with the nucleus.
      \label{fig:atom}
    }
\end{figure}
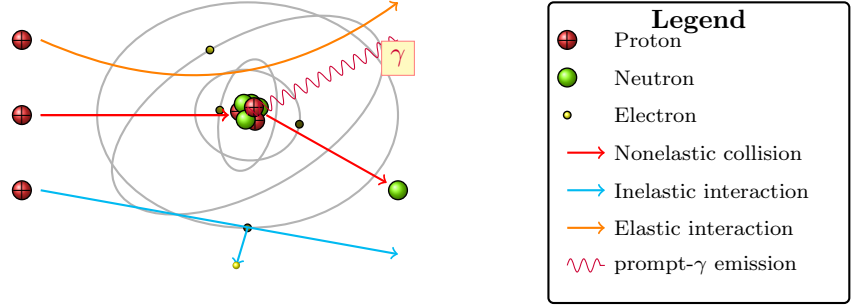

Inelastic collisions with electrons lead to gradual energy loss,
typically described deterministically through the Bragg-Kleeman or
Bethe-Bloch equation~\cite{ICRU49}, and are responsible for the sharp
rise and falloff of the Bragg peak. However, due to the stochastic
nature of these collisions, protons of the same initial energy do not
all stop at the same depth, leading to an inherent longitudinal spread
known as range straggling~\cite{bortfeld1997analytical}. Angular
deflections arise primarily from elastic Coulomb interactions with
nuclei and are typically small but frequent, resulting in lateral beam
broadening through multiple scattering~\cite{Gottschalk1993}. These
deflections are well described by Brownian motion on the
sphere. Occasional inelastic nuclear interactions produce discrete
energy losses and secondaries such as neutrons, contributing to the
dose halo and distal falloff~\cite{Schneider2002}, but are not
included in the present model, although further details can be found
in \cite{crossley2024jump}.

Our focus is therefore on a simplified yet representative regime. This
leads to a coupled system for $(X_t, E_t, \Omega_t)$ that captures the
dominant physics while remaining analytically and numerically
tractable.

The dynamics are described by the stochastic differential equations
\begin{equation}
  \label{eq:SDE1}
  \begin{split}
    \d X_t &= \Omega_t \d t \\
    \d E_t &= -S(E_t) \d t \\
    \d \Omega_t &= \sqrt{2\epsilon(E_t)} \circ \d W_t^S.
  \end{split}
\end{equation}
The first equation expresses free streaming with unit velocity in the
direction $\Omega_t$. The second describes the deterministic decrease
in energy, governed by the stopping power $S(E)$. The third models
angular diffusion, where $W_t^S$ is a standard Brownian motion on the
unit sphere and $\epsilon(E)$ determines the scattering strength. In
the simplest case, $\epsilon(E) = \epsilon_0$ may be taken constant,
though in practice it is energy dependent.

The angular SDE is interpreted in the Stratonovich sense, consistent
with Brownian motion on the sphere. This avoids explicit curvature
drift terms and preserves the uniform distribution as the angular
equilibrium law. 

The stochastic description is naturally linked to a forward PDE for
the particle density $\psi(t, \vec x, \vec \omega, e)$ in phase
space. This fluence satisfies
\begin{equation}
  \frac{\partial \psi}{\partial t}
  +
  \vec \omega \cdot \nabla_{\vec x} \psi
  -
  \frac{\partial}{\partial e} \qp{ S(e) \psi }  
  -
  \epsilon(e) \Delta_{\vec \omega} \psi = 0,
\end{equation}
with inflow boundary condition
\begin{equation}
  \psi = g \quad \text{on } \Gamma_-,
\end{equation}
where
\begin{equation}
  \Gamma_- = \{ (\vec x, \vec \omega, e) \in \partial D \times
  \mathbb{S}^2 \times \reals^+ : \vec \omega \cdot \vec n(\vec x) <
  0 \}
\end{equation}
and $\Delta_{\vec \omega}$ denotes the Laplace-Beltrami operator on
the sphere. By the Feynman-Kac formula
\cite{del2004feynman,kyprianou2025unified}, $\psi$ admits a
probabilistic representation in terms of expectations over
realisations of the stochastic process \eqref{eq:SDE1}.

\begin{remark}[Track-length parametrisation]
  \autoref{eq:SDE1} formulates the dynamics with respect to track
  length rather than physical time. Under a physical-time
  parametrisation, the transport equation for $X_t$ would involve a
  velocity factor depending on the current energy $E_t$, rather than
  the unit speed assumed here. The two formulations are related by a
  time-change transformation, so the choice of parametrisation does
  not alter the underlying physics but simplifies the mathematical
  structure.
\end{remark}

\subsection{Stopping Power Models}

The stopping power $S(E)$ characterises the mean rate of energy loss
per unit path length. Microscopically it is described by the
Bethe-Bloch equation \cite{ICRU49}, which accounts for inelastic
collisions with atomic electrons. For reduced models, the
Bragg-Kleeman law provides a simpler approximation,
\begin{equation}
  S(E) = \frac{1}{p \alpha} E^{1-p},
\end{equation}
where $\alpha > 0$ and $p \in [1,2]$ are material-dependent
parameters. This form captures the essential energy dependence of
proton stopping power and is widely used in analytic descriptions of
range and dose \cite{ashby2025efficient}.

\subsection{Moli\'ere Theory}
\label{subsec:moliere}

The angular diffusion coefficient $\epsilon(E)$ in~\eqref{eq:SDE1}
reflects the cumulative deflections from multiple small-angle Coulomb
scatterings. Physical theory and empirical evidence both support its
dependence on the proton’s energy. In particular, Moli\'ere's theory
of multiple Coulomb scattering~\cite{moliere1948theorie} predicts that
the scattering power is inversely related to the square of the
particle's momentum.

In the small-angle limit, the mean squared scattering angle per unit
path length satisfies
\begin{equation}
  \frac{\d}{\d x} \langle \theta^2 \rangle \propto \frac{1}{\beta^2 p^2},
\end{equation}
where $\beta = v/c$ is the dimensionless velocity and $p$ is the
momentum. For protons in the therapeutic energy range, a
non-relativistic approximation gives $p \sim \sqrt{2mE}$ and $\beta^2
\sim E/m$, yielding
\begin{equation}
  \epsilon(E) \propto \frac{1}{E^2}.
\end{equation}
This energy dependence implies that angular diffusion intensifies at
lower energies, consistent with the observed lateral beam broadening
in the distal region of a proton beam.

An empirical parametrisation inspired by the Highland
formula~\cite{highland1975some} is
\begin{equation}\label{eq:moliere-coef}
  \epsilon(E) = \frac{\bar{\epsilon}}{E^2 + \epsilon_c^2},
\end{equation}
where $\bar{\epsilon}$ is a reference scattering strength and
$\epsilon_c$ is a regularising constant, typically in the range 5-10
MeV, calibrated to experimental beam spread.

\subsection{Range Straggling}

The energy loss of a charged particle in a medium is subject to
statistical fluctuations due to the discrete nature of ionisation
interactions. This variability, known as \textit{range straggling},
leads to deviations in the stopping range of particles with the same
initial energy. The function $T(E)$ characterises the variance of the
energy loss per unit path length and is derived from statistical
models of energy deposition.

The first statistical description was provided by Bohr
\cite{bohr1913xxxvii}, who showed that for many small, independent
collisions, the variance of energy loss grows proportionally to the
stopping power. Landau \cite{landau1944} refined this picture by
accounting for the asymmetric tail of the energy-loss distribution. In
the thick-target regime relevant to therapy beams, a standard
modelling choice is
\begin{equation}
  \label{eq:straggling}
  T(E) = \kappa S(E) E,
\end{equation}
where $\kappa$ is a medium-dependent coefficient. For the
Bragg-Kleeman form of $S(E)$, one obtains
\begin{equation}
  T(E) = \frac{\kappa}{p\alpha} E^{2-p}.
\end{equation}

The range variance then follows from
\begin{equation}
  \operatorname{Var}(R)
  =
  \int_0^{E_0} \frac{T(E)}{S(E)^2} \d E
  =
  \kappa p\alpha \int_0^{E_0} E^{p} \d E
  =
  \kappa p\alpha \frac{E_0^{p+1}}{p+1}.
\end{equation}
Solving for $\kappa$ gives
\begin{equation}
  \kappa
  =
  \frac{(p+1) \operatorname{Var}(R)}{p\alpha E_0^{p+1}}.
\end{equation}
Based on the measured distal $90\%-10\%$ falloff width of $1.7$mm in
water for a $62$ MeV beam at Clatterbridge \cite{bonnett199362}, the
standard deviation in proton range can be estimated as $\sigma_R
\approx 0.072$ cm. Substituting $\sigma_R^2 = \operatorname{Var}(R)$
with $p = 1.77$, $\alpha = 2.2 \times 10^{-3}$, $E_0 = 62$ and
$\sigma_R = 0.072$, we obtain
\begin{equation}
  \kappa \approx \frac{ (0.072)^2 \cdot (1.77 + 1) } { 1.77 \cdot (2.2
    \times 10^{-3}) \cdot 62^{2.77} } \approx 4\times 10^{-5}.
\end{equation}

In the distal falloff of the Bragg peak, nuclear interactions also
contribute to broadening the range distribution. While we do not model
these effects here, a description of their impact is given in
\cite{crossley2024jump}.

\subsection{Modification of the SDE to Account for Range Straggling}

The deterministic energy loss model in~\eqref{eq:SDE1} neglects
fluctuations arising from the stochastic nature of energy-depositing
collisions. These fluctuations lead to \emph{range straggling}, where
protons of the same initial energy exhibit a spread in stopping
distance. To account for this, we introduce a stochastic term in the
energy equation, with variance governed by the empirical straggling
function $T(E)$. The resulting stochastic differential equation reads
\begin{equation}
  \label{eq:SDE2}
  \begin{split}
    \d X_t &= \Omega_t \d t, \\
    \d E_t &= -S(E_t) \d t + \sqrt{T(E_t)} \d W_t^E, \\
    \d \Omega_t &= \sqrt{2\epsilon(E_t)} \circ \d W_t^S,
  \end{split}
\end{equation}
where $W_t^E$ is a standard Brownian motion in energy. This choice
ensures that the energy variance grows in accordance with experimental
observations of range straggling in water phantoms. The deterministic
term continues to govern the average energy loss, while the stochastic
term introduces variability that accounts for deviations in stopping
range between different protons.

As a result of this modification, the corresponding transport equation
includes a second derivative in energy, leading to
\begin{equation}
  \frac{\partial \psi}{\partial t}
  +
  \vec \omega \cdot \nabla_{\vec x} \psi
  -
  \frac{\partial}{\partial e} \qp{ S(e) \psi }
  -
  \frac{1}{2} \frac{\partial^2}{\partial e^2} \qp{ T(e) \psi }
  -
  \epsilon(e) \Delta_{\vec \omega} \psi = 0.
\end{equation}
This additional term captures the spread in energy due to range
straggling and contributes to the overall uncertainty in proton range.
Its inclusion is particularly relevant in clinical proton therapy,
where range uncertainties impact treatment planning and dose
deposition.

We note that the coefficients $S(E)$, $T(E)$ and $\epsilon(E)$ are
smooth and locally Lipschitz on $(0,\infty)$, ensuring well-posedness
of the SDE system up to the point where $E_t \to 0$. One way of
handling the delicate behaviour near zero is to impose a killing
boundary at some small $E=E_{\min}>0$. By imposing a killing boundary
at $E=E_{\min}$, we need only consider the process on the domain
$[E_{\min},\infty)$, on which all coefficients are globally Lipschitz.

The form (\ref{eq:straggling}) neglects the asymmetric Landau tail in
the energy-loss distribution and is most accurate in the thick-target
regime where many small collisions dominate. Alternative
parameterisations, such as those based on Vavilov theory
\cite{noshad2012investigation}, provide higher fidelity when rare
large energy losses are important.

\subsection{Stopping Time and Particle Killing}
\label{sec:killing-boundary}
We define the stochastic process
\begin{equation}
  \{ \qp{ X_t, \Omega_t, E_t} :  t \ge 0 \},
\end{equation}
where $X_t \in D\subset \reals^3$, $\Omega_t \in \mathbb{S}^2$ and
$E_t \in \reals^+$ denote the proton's spatial position, direction and
energy at time $t$, respectively. The process models the evolution of
an individual proton trajectory in the phase space $D \times
\mathbb{S}^2 \times \reals^+$.

The process is defined up to a random stopping time $\tau$, given
by
\begin{equation}
  \tau := \inf\{t > 0 : E_t = {E_{\min}} \text{ or } X_t \notin D \},
\end{equation}
which corresponds to the track length at which the proton either exits
the spatial domain or loses all its energy. The stopping time $\tau$
is a random variable that depends on the initial configuration $(X_0,
\Omega_0, E_0)$, though we suppress this dependence to simplify
notation.

After the stopping time $\tau$, the proton is considered to have been
absorbed or exited the domain and no longer evolves. To formalise
this, we extend the state space by adding a cemetery state $\dagger$,
and define the process over $[0, \infty)$ by
  \begin{equation}
    \qp{ X_t, \Omega_t, E_t}
    := 
    \begin{cases}
      \qp{ X_t, \Omega_t, E_t}, & t < \tau, \\
      \dagger, & t \geq \tau.
    \end{cases}
  \end{equation}
  This extension ensures that the particle is well defined for all $t
  \geq 0$, even after the proton has exited the physical system.  The
  cemetery state $\dagger$ is absorbing, i.e. if the process enters
  $\dagger$ at time $\tau$, then $(X_t,\Omega_t,E_t) = \dagger$ for
  all $t \ge \tau$. For consistency, we define all test functions $f:
  (D \times \mathbb{S}^2 \times [E_{\min},\infty)) \cup {\dagger} \to
    \mathbb{R}$ such that $f(\dagger) = 0$. In this way, the state
    $\dagger$ contributes nothing to integrals, expectations or
    occupation measures. This convention allows us to work on the
    extended domain while effectively restricting calculations to the
    live configuration space $D\subset \reals^3$, $\Omega_t \in
    \mathbb{S}^2$.

\subsection{Logarithmic Transformation for Energy Positivity}

In modelling proton transport, the energy variable $E_t$ must remain
strictly positive. Direct numerical simulation of its stochastic
dynamics can nevertheless produce negative values due to
discretisation error, especially when the diffusion coefficient is
state-dependent.

Even in the deterministic case without range straggling, the energy
equation exhibits a finite stopping time. For $p \in (1,2)$, the
solution
\begin{equation}
  E(t) = \bigl( E_0^p - t/\alpha \bigr)^{1/p}
\end{equation}
is strictly positive only up to the critical time $T = \alpha
E_0^p$. As $t \to T^-$, the solution remains continuous but becomes
non-differentiable as $E \to 0^+$. This illustrates the delicate
behaviour near $E=0$, even without stochasticity.

When range straggling is included, the SDE for $E_t$ contains both a
degenerate diffusion coefficient and a singular drift. Neither is
locally Lipschitz at $E=0$, violating standard conditions for
existence and uniqueness of strong solutions. While the probability of
$E_t$ becoming negative may remain small for moderate times, global
well-posedness is not guaranteed.

To regularise the problem and enforce positivity, we reformulate the
dynamics in terms of a log-transformed energy variable. Such
transformations are standard in the SDE literature as a means of
stabilising dynamics and constraining state variables to their
physical domains (see, e.g., \cite{yi2021positivity}).

We define the logarithmic variable
\begin{equation}
  Y_t = \log E_t.
\end{equation}
Applying Itô's lemma, we compute
\begin{equation}
  \d Y_t
  =
  \frac{1}{E_t} \d E_t
  -
  \frac{1}{2} \frac{1}{E_t^2} (\d E_t)^2.
\end{equation}
Substituting the SDE for $E_t$,
\begin{equation}
  \d E_t
  =
  -S(E_t) \d t
  +
  \sqrt{T(E_t)} \d W_t^E,
\end{equation}
we obtain
\begin{equation}\label{eq:log-energy}
  \d Y_t
  =
  \frac{1}{E_t} \qp{
    -S(E_t) \d t
    +
    \sqrt{T(E_t)} \d W_t^E }
  -
  \frac{1}{2} \frac{1}{E_t^2} T(E_t) \d t.
\end{equation}
Since $E_t = \exp\qp{Y_t}$, we rewrite the equation entirely in terms of $Y_t$:
\begin{equation}
  \d Y_t
  =
  \qp{
    -S(\exp\qp{Y_t}) \exp\qp{-Y_t}
    -
    \frac{1}{2} T(\exp\qp{Y_t}) \exp\qp{-2Y_t} }
  \d t
  +
  \sqrt{T(\exp\qp{Y_t})} \exp\qp{-Y_t} \d W_t^E.
\end{equation}
This reformulation ensures that $Y_t$ evolves naturally in
$\mathbb{R}$, so its numerical approximation cannot lead to negative
values of $E_t$ when transformed back via
\begin{equation}
  E_t = \exp\qp{Y_t}.
\end{equation}

\subsection{Structural Properties of the Model}
\label{sec:structures}

The system evolves in terms of $(X_t, Y_t, \Omega_t)$, and is given by
\begin{equation}
  \label{eq:LogSDE}
  \begin{split}
    \d X_t &= \Omega_t   \d t, \\
    \d Y_t
    &= \qp{ -S(\exp\qp{Y_t}) \exp\qp{-Y_t} - \tfrac{1}{2} T(\exp\qp{Y_t}) \exp\qp{-2Y_t} } \d t
    +
    \sqrt{T(\exp\qp{Y_t})} \exp\qp{-Y_t} \d W_t^E, \\
    \d \Omega_t &= \sqrt{2\epsilon(\exp\qp{Y_t})} \circ \d W_t^S,
  \end{split}
\end{equation}
with initial conditions
\begin{equation}
  X_0 = x_0 \in D\subset \mathbb{R}^3, \quad \Omega_0 = \omega_0 \in \mathbb{S}^2, \quad Y_0 = \log E_0, \quad E_0 > 0.
\end{equation}

\begin{proposition}[Energy positivity under the log transform]
  Let $(X_t,Y_t,\Omega_t)$ solve \eqref{eq:LogSDE} with $E_0>0$ and
  define $E_t := \exp\qp{Y_t}$ for $t<\tau$, where $\tau$ is the killing time
  from Section~\ref{sec:killing-boundary}. Then $E_t>0$ almost surely
  for all $t<\tau$.
\end{proposition}

\begin{proof}
  By construction $Y_t\in\mathbb{R}$ with continuous sample paths for
  $t<\tau$ (the coefficients of \eqref{eq:LogSDE} are locally Lipschitz
  on the live domain), hence $E_t=\exp(Y_t)$ is well defined and
  strictly positive for all $t<\tau$, almost surely. After $\tau$ the
  process is sent to the cemetery state and $E_t$ is not a physical
  variable.
\end{proof}

\begin{proposition}[Monotone decay of the expected energy]
  Let $(X_t,Y_t,\Omega_t)$ solve \eqref{eq:LogSDE} with killing time
  $\tau$ and set $E_t = \exp\qp{Y_t}$ for $t<\tau$. Suppose $S(E)>0$ for
  all $E>0$ and $\mathbb{E} \left[ S(E_t) \right]<\infty$ for
  $t<\tau$. Then for all $t<\tau$,
  \begin{equation}
    \frac{\d}{\d t} \mathbb{E}[E_t] = \mathbb{E} \qb{ -S(E_t) } < 0
  \end{equation}
  with strict inequality whenever $\mathbb{P}(t<\tau)>0$.
\end{proposition}

\begin{proof}
  Applying Itô’s formula to $E_t = \exp\qp{Y_t}$, we have
  \begin{equation}
    \d E_t = \exp\qp{Y_t} \d Y_t + \tfrac{1}{2} \exp\qp{Y_t} (\d
    Y_t)^2.
  \end{equation}
  Substituting the SDE for $Y_t$ from \eqref{eq:LogSDE}, we
  write
  \begin{align}
    \d Y_t &= \mu_t \d t + \beta_t \d W_t^E,
    \\
    \mu_t &= -S(\exp\qp{Y_t}) \exp\qp{-Y_t} - \tfrac{1}{2} T(\exp\qp{Y_t}) \exp\qp{-2Y_t},
    \\
    \beta_t &= \sqrt{T(\exp\qp{Y_t})} \exp\qp{-Y_t}.
  \end{align}
  Then the drift of $E_t$ becomes
  \begin{equation}
    \frac{\d}{\d t}
    \mathbb{E}[E_t]
    =
    \mathbb{E} \qb{ \exp\qp{Y_t} \mu_t + \tfrac{1}{2} \exp\qp{Y_t} \beta_t^2}.
  \end{equation}
  Substituting the expressions for $\mu_t$ and $\beta_t$, we
  compute
  \begin{align}
    \exp\qp{Y_t} \mu_t &= -S(\exp\qp{Y_t}) - \tfrac{1}{2} T(\exp\qp{Y_t}) \exp\qp{-Y_t},
    \\
    \tfrac{1}{2} \exp\qp{Y_t} \beta_t^2 &= \tfrac{1}{2} T(\exp\qp{Y_t})
    \exp\qp{-Y_t},
  \end{align}
  so the $T$ terms cancel and we obtain
  \begin{equation}
    \frac{\d}{\d t} \mathbb{E}[E_t]
    =
    -\mathbb{E} \qb{ S(E_t) },
  \end{equation}
  which is strictly negative for all $t$ as $S(E_t) > 0$.
\end{proof}
  
\begin{proposition}[Stationarity and exponential mixing for constant angular diffusion]
\label{prop:ergodic}
Let $\Omega_t \in \mathbb{S}^2$ solve
\begin{equation}
  \d\Omega_t = \sqrt{2 \epsilon_0}\circ \d W_t^S,
\end{equation}
with constant $0 < \epsilon_0 \in \reals$, where $W_t^S$ is Brownian
motion on the sphere. Let $\sigma$ denote the uniform probability
measure on $\mathbb{S}^2$. Then $\sigma$ is invariant and
ergodic. Furthermore, writing the Markov semigroup as
\begin{equation}
  (P_t f)(\vec \omega)
  :=
  \mathbb{E}\qp{ f(\Omega_t) \ \vert \ \Omega_0 = \vec \omega },
\end{equation}
we have, for all $f \in \leb 2(\mathbb{S}^2,\sigma)$ with
$\int_{\mathbb{S}^2} f \d \sigma = 0$,
\begin{equation}\label{eq:L2-mixing}
  \Norm{P_t f}_{\leb 2(\sigma)}
  \le
  \exp\qp{-2\epsilon_0 t} \Norm{f}_{\leb 2(\sigma)},
\end{equation}
so the law of $\Omega_t$ converges exponentially fast to $\sigma$ as
$t\to\infty$ in $\leb 2$. Equivalently, if $\rho_t$ is the density of
the law of $\Omega_t$ with respect to $\sigma$, then
\begin{equation}
  \Norm{\rho_t - 1}_{\leb 2(\sigma)}
  \le
  \exp\qp{-2 \epsilon_0 t}
  \Norm{\rho_0 - 1}_{\leb 2(\sigma)} .
\end{equation}
\end{proposition}

\begin{proof}
  The Stratonovich dynamics have generator
  \begin{equation}
    L=\epsilon_0 \Delta_{\vec \omega}
  \end{equation}
  on defined on $C^\infty(\mathbb{S}^2)$, which is dense in
  $L^2(\mathbb{S}^2,\sigma)$. Integration by parts shows that
  $\Delta_{\vec \omega}$ is symmetric with respect to the
  $L^2(\sigma)$ inner product and since $\Delta_{\vec \omega} 1 = 0$,
  the uniform measure $\sigma$ is invariant under the semigroup
  $\exp\qp{tL}$. Now, the spectrum of $-\Delta_{\vec \omega}$ on
  $\mathbb{S}^2$ is $\{ \ell(\ell+1) : \ell=0,1,2,\dots \}$ with
  eigenspaces given by spherical harmonics and the first non-zero
  eigenvalue being $2$. Therefore, on the mean-zero subspace, the
  semigroup $P_t = \exp\qp{tL}$ contracts as
  \eqref{eq:L2-mixing}. Compactness and ellipticity imply uniqueness
  of the invariant measure and ergodicity.
\end{proof}

\begin{remark}[Energy-dependent scattering]
When the angular coefficient depends on energy, the joint process obeys
\begin{equation}
  \d\Omega_t
  = \sqrt{2\epsilon(E_t)}\circ \d W_t^S,
\end{equation}
so the \emph{marginal} dynamics of $\Omega_t$ are time-inhomogeneous
and generally non-Markovian. For each frozen energy level $e$, the
uniform law on $\mathbb{S}^2$ is invariant for the generator
$\epsilon(e)\Delta_{\vec\omega}$ (since $\Delta_{\vec\omega}1=0$), but
there need not exist a stationary measure for the marginal of
$\Omega_t$ when coupled to $E_t$. If $\epsilon(E)$ grows as
$E\downarrow 0$ (as suggested by Moli\'ere theory), the instantaneous
angular mixing rate increases along paths as the energy decreases and,
informally, conditional on survival, the angular distribution becomes
rapidly isotropic.
\end{remark}

\begin{remark}[Long-time behaviour prior to killing]
Let $\tau$ be the killing time from \S\ref{sec:killing-boundary}. On
$t<\tau$, the components of \eqref{eq:LogSDE} exhibit distinct
trends. The energy decreases on average and, under mild regularity,
$E_t\to 0$ almost surely as $t\uparrow \tau$. If $\epsilon(E)$
increases as $E\downarrow 0$, the angular motion mixes faster and the
conditional law of $\Omega_t$ approaches isotropy before killing. The
position exhibits increasing lateral spread due to angular diffusion
with killing at $E_{\min}>0$, $X_t$ is only defined on $[0,\tau)$,
  after which the process is sent to the cemetery state.
\end{remark}

\section{Euler--Maruyama Discretisation}
\label{sec:euler}

To solve the stochastic system numerically, we begin by applying the
Euler--Maruyama method, an explicit scheme for approximating solutions
to SDEs. While straightforward to implement, this scheme fails to
preserve key structural properties of the continuous dynamics,
motivating the development of geometric alternatives in the next
section.

\subsection{Time Discretisation}

Let $\{t_n\}_{n=0}^{N}$ be a uniform discretisation of the time
interval $[0,T]$ with step size $h$, where
\begin{equation}
  t_n = n h, \quad \text{for } n = 0,1,\dots,N, \quad \text{with } h = \frac{T}{N}.
\end{equation}
We denote the discrete approximations of $X_t$, $Y_t$ and $\Omega_t$
at time $t_n$ by $X_n$, $Y_n$ and $\Omega_n$, respectively. The
corresponding energy variable is given by $E_n = \exp\qp{Y_n}$.

\subsection{Euler--Maruyama Approximation}

Applying the Euler--Maruyama scheme to the system \eqref{eq:LogSDE},
we obtain the updates
\begin{equation}
  \begin{split}
    X_{n+1} &= X_n + h \Omega_n, \\
    Y_{n+1} &= Y_n
      - h S(\exp\qp{Y_n}) \exp\qp{-Y_n}
      - \tfrac{h}{2} T(\exp\qp{Y_n}) \exp\qp{-2Y_n} 
      + \sqrt{T(\exp\qp{Y_n})} \exp\qp{-Y_n} \xi_n^E, \\
      \Omega_{n+1} &= \qp{1 - \epsilon_0 h} \Omega_n + \sqrt{2\epsilon_0}  \xi_n^S.     
  \end{split}
\end{equation}
Here, $\xi_n^E$ denotes a Gaussian increment with distribution
$\mathcal{N}(0, h)$ and the angular increment $\xi_n^S$ is
constructed by projecting standard Gaussian noise onto the tangent
space at $\Omega_n$, that is,
\begin{equation}
  \xi_n^S
  =
  \qp{ I - \Omega_n \Omega_n^\top } \eta_n, \quad \eta_n \sim \mathcal{N}(0, h I_3).
\end{equation}

\subsection{Properties and Convergence Notes}
The Euler--Maruyama scheme is a first-order method that achieves
strong convergence of order $O(h^{1/2})$ \footnote{we use 'strong
order $p$' in the standard mean-square sense} for SDEs with globally
Lipschitz coefficients~\cite{kloeden1992stochastic}. In our setting,
the energy equation involves nonlinear drift and diffusion terms that
are not globally Lipschitz in $E_t$, due to their power-law form. This
precludes direct application of standard convergence theory.

Imposing a killing boundary at $E=E_{\min}>0$ restricts the state
space to a compact subdomain on which the coefficients relevant for
the finite-time analysis below are Lipschitz. This is the setting in
which we establish the $O(h^{1/2})$ convergence rate in practice.

\begin{remark}[Loss of Structure under Euler Discretisation]
  While Euler--Maruyama is consistent, it does not preserve key
  geometric properties of the SDE system \eqref{eq:LogSDE}.

  For the angular dynamics, the update does not maintain the unit-norm
  constraint, so the iterates drift off the sphere. A post-processing
  projection restores the constraint, but alters the discretised
  dynamics and is therefore best regarded as a baseline correction
  rather than a structure-preserving approximation.
    
  For the energy dynamics, direct discretisation of $E_t$ can also
  produce negative energies, since the Gaussian increment may outweigh
  the deterministic drift. This violates the strict positivity of
  physical energy and can lead to instability. The log-transformed
  formulation avoids this issue by evolving $Y_t=\log E_t$ in
  $\mathbb{R}$, ensuring that $E_n = \exp\qp{Y_n}$ remains strictly
  positive at all times.
\end{remark}

\subsection{Geometric Angular Update}
\label{sec:geo-euler}

A common workaround (see e.g. \cite{stoltz2010free}) for
norm-constraint violation is to apply an explicit normalisation after
each Euler update, that is post-process
\begin{equation}
  \Omega_{n+1} = \qp{1 - \epsilon_0 h}\Omega_n + \sqrt{2\epsilon_0} \xi_n^S,
  \quad
  \Omega_{n+1} \leftarrow \frac{\Omega_{n+1}}{\norm{\Omega_{n+1}}}.
\end{equation}
While this projection enforces norm preservation, it introduces a
statistical bias as the renormalisation step distorts the distribution
of $\Omega_n$ and breaks consistency with the continuous SDE.

A more consistent approach is to update $\Omega_n$ using the
exponential map on the sphere
\begin{equation}
  \Omega_{n+1} = \Exp{\Omega_n} \qp{ \sqrt{2\epsilon_0} \xi_n^S },
\end{equation}
where $\Exp{\Omega_n}$ denotes the Riemannian exponential map on the
sphere. This operation transports $\Omega_n$ along the geodesic
determined by $\xi_n^S$, ensuring that the update remains on the
sphere while preserving the correct diffusion behaviour.

\begin{remark}[Exact Simulation of Angular Process with Constant $\epsilon$]
  When $\epsilon=\epsilon_0$ is constant the process admits a symmetry
  reduction that allows exact sampling from the transition
  law~\cite{mijatovic2020note,beskos2005exact}.  Specifically, by
  rotating the coordinate system so that the initial direction aligns
  with the north pole, the process reduces to a time-homogeneous
  Wright--Fisher diffusion for the polar angle. This permits exact
  simulation of $\Omega_t$ by drawing from known distributions on
  $\mathbb{S}^2$, conditioned on the elapsed time and initial
  orientation. This method provides exact marginal samples and
  bypasses the need for time stepping when only the angular position
  at fixed time is required \cite{crossley2024jump}. For
  energy-dependent $\epsilon(E)$, in general no such reduction is
  available and time discretisation is necessary.
\end{remark}

\begin{theorem}[Strong Convergence of the Geometric Euler Scheme]
  Let $(X_t, Y_t, \Omega_t)$ be the solution to the log-transformed
  SDE system \eqref{eq:LogSDE}, with $E_t = \exp\qp{Y_t}$. Let $(X_n, Y_n,
  \Omega_n)$ denote the numerical approximation given by the geometric
  Euler scheme
  \begin{equation}
    \label{eq:LogEulerExp}
    \begin{split}
      X_{n+1} &= X_n + h  \Omega_n, \\
      Y_{n+1} &=
        Y_n
        - h S(\exp\qp{Y_n}) \exp\qp{-Y_n} 
        - \tfrac{h}{2} T(\exp\qp{Y_n}) \exp\qp{-2Y_n} 
        + \sqrt{T(\exp\qp{Y_n})} \exp\qp{-Y_n}  \xi_n^E, \\
      \Omega_{n+1} &= \Exp{\Omega_n} \qp{ \sqrt{2\epsilon(\exp\qp{Y_n})}  \xi_n^S },
    \end{split}
  \end{equation}
  where $\xi_n^E \sim \mathcal{N}(0,h)$ is a scalar Gaussian increment
  for the energy noise, and $\xi_n^S \sim \mathcal{N}(0, h I_3)$ is a
  tangent vector sampled in $T_{\Omega_n} \mathbb{S}^2$. The map
  $\Exp{\Omega_n}$ denotes the exponential map on the sphere. Then the
  scheme satisfies the strong convergence estimate
  \begin{equation}
    \mathbb{E} \qb{
      \sup_{0 \leq n \leq N}
      \qp{
        \norm{E_n - E(t_n)}^2
        +
        \norm{\Omega_n - \Omega(t_n)}^2
      }
    }
    \leq C h,
  \end{equation}
  and
  \begin{equation}
    \mathbb{E} \qb{
      \sup_{0 \leq n \leq N}
      \qp{
            \norm{X_n - X(t_n)}^2
      }
    }
    \leq C h^2
  \end{equation}
  where $E_n = \exp\qp{Y_n}$ and $C$ is a constant independent of $h$. That
  is, the method converges strongly with order $O(h^{1/2})$.
\end{theorem}

\begin{proof}
  We begin with the energy variable. The drift and diffusion
  coefficients in the $Y_t$ SDE are given by smooth compositions of
  the form $S(\exp\qp y)\exp\qp{-y}$ and $T(\exp\qp y)\exp\qp{-2y}$, which are smooth and
  locally Lipschitz for $y \in \mathbb{R}$ provided $T > 0$ since
  $E_{min} > 0$. Although they do not satisfy global linear growth
  conditions due to their exponential decay as $y \to -\infty$, the
  solution $Y_t$ remains bounded with high probability on finite time
  intervals, since the energy process $E_t = \exp\qp{Y_t}$ decays over
  time.

  Therefore, the Euler--Maruyama scheme applied to $Y_t$ converges
  strongly with order $O(h^{1/2})$, by standard localisation arguments
  (see \cite[Thm~10.2.2]{kloeden1992stochastic}). We obtain
  \begin{equation}
    \mathbb{E} \qb{ \sup_{0 \leq n \leq N} \norm{Y_n - Y(t_n)}^2 } \leq C h,
  \end{equation}
  and thus, by smoothness of the exponential,
  \begin{equation}
    \mathbb{E} \qb{ \sup_{0 \leq n \leq N} \norm{E_n - E(t_n)}^2 } \leq C h.
  \end{equation}
   
  For the angular variable we proceed by estimating the error via the
  identity
\begin{equation}
  \Omega_{n+1} - \Omega(t_{n+1})
  =
  \qb{ \Exp{\Omega_n} \qp{ \sqrt{2\epsilon(E_n)} \xi_n^S}
    -
    \Exp{\Omega(t_n)} \qp{ \sqrt{2\epsilon(E(t_n))} \xi_n^S}
  }
  + \delta_n,
\end{equation}
where $\delta_n = \Exp{\Omega(t_n)} \qp{ \sqrt{2\epsilon(E(t_n))}
\xi_n^S } - \Omega(t_{n+1})$ denotes the local discretisation
error. The first term is controlled via Lipschitz continuity of the
exponential map in both base point and argument
\begin{equation}
  \norm{
    \Exp{\Omega_n} \qp{
      \sqrt{2\epsilon(E_n)} \xi_n^S}
    -
    \Exp{\Omega(t_n)} \qp{
      \sqrt{2\epsilon(E(t_n))} \xi_n^S} }
  \leq
  C \qp{ \norm{\Omega_n - \Omega(t_n)} + \norm{E_n - E(t_n)}} \norm{ \xi_n^S}.
\end{equation}
Taking expectations and using $\norm{ \xi_n^S} = O(h^{1/2})$ in mean
square, and $\norm{E_n - E(t_n)} = O(h^{1/2})$, we obtain
\begin{equation}
  \mathbb{E} \qb{ \norm{\Omega_{n+1} - \Omega(t_{n+1})}^2 }
  \leq
  (1 + C h) \mathbb{E} \qb{ \norm{\Omega_n - \Omega(t_n)}^2 } + C h^2.
\end{equation}
By Gr\"onwall's inequality, this implies
\begin{equation}
  \mathbb{E} \qb{ \sup_{0 \leq n \leq N} \norm{\Omega_n - \Omega(t_n)}^2 } \leq C h.
\end{equation}

Finally, the position update is deterministic and, upon noting
\begin{equation}
  X(t_{n+1}) = X(t_n) + \int_{t_n}^{t_{n+1}} \Omega(s) \d s,
\end{equation}
subtracting and taking norms,
\begin{equation}
  \norm{X_{n+1} - X(t_{n+1})}
  \leq
  \norm{X_n - X(t_n)}
  +
  h \norm{\Omega_n - \Omega(t_n)}
  +
  \norm{ \int_{t_n}^{t_{n+1}} \qp{ \Omega(t_n) - \Omega(s) } \d s }.
\end{equation}
The last term is $O(h^{3/2})$ in mean square due to H\"older
continuity of $\Omega_t$. Using the previous bound on the angular
error, we conclude
\begin{equation}
  \mathbb{E} \qb{ \sup_{0 \leq n \leq N} \norm{X_n - X(t_n)}^2} \leq C h^2,
\end{equation}
completing the proof.
\end{proof}

\begin{remark}[Structure preservation under the geometric Euler scheme]
  The log-transformed geometric Euler scheme \eqref{eq:LogEulerExp}
  preserves the essential constraints of the continuous model
  \eqref{eq:LogSDE}. Since the energy is evolved through the logarithmic
  variable, the reconstructed iterates satisfy
  \begin{equation}
    E_n=\exp\qp{Y_n}>0
    \qquad\text{for all } n,
  \end{equation}
  so positivity is maintained without projection or truncation. The
  angular update is performed by the exponential map on
  $\mathbb{S}^2$, and therefore
  \begin{equation}
    \Omega_n\in\mathbb{S}^2
    \qquad\text{for all } n,
  \end{equation}
  by construction. Thus the numerical trajectory remains in the
  physically admissible state space throughout the computation. These
  are the principal structural features required for the finite-time
  pathwise approximation considered in this paper.
\end{remark}

\subsection{Illustrative Numerical Example - Angular Drift in Two Dimensions}

To highlight the structural deficiencies of the naive Euler scheme, we
consider angular diffusion on the unit circle $\mathbb{S}^1$,
discretised over a long time interval. We compare the standard Euler
update
\begin{equation}
  \Omega_{n+1}
  =
  \Omega_n\qp{1 - \epsilon h} + \sqrt{2\epsilon h} \xi_n, \quad \xi_n \sim \mathcal{N}(0, I_2),
\end{equation}
applied both with and without explicit renormalisation, against the
structure-preserving update based on the exponential map. In
two dimensions, this corresponds to a rotation
\begin{equation}
  \Omega_{n+1}
  =
  \operatorname{Rot}(\theta_n) \Omega_n,
  \quad
  \theta_n = \sqrt{2\epsilon h}   \lambda_n,
  \quad
  \lambda_n \sim \mathcal{N}(0,1).
\end{equation}
The naive Euler scheme accumulates norm error and introduces bias in
the angular distribution, particularly in long-time simulations. In
contrast, the exponential map method preserves the unit-norm
constraint exactly and reproduces the expected uniform distribution in
angle.

Figure~\ref{fig:angle_drift} shows representative trajectories and
histograms of sampled angles after $N = 10^5$ steps with time step $h
= 10^{-2}$. The difference in long-time behaviour is clearly visible.

\section{Higher-Order Structure-Preserving Schemes}
\label{sec:high-order}

For applications requiring higher accuracy, we now construct a strong
order-1 structure-preserving schemes for both the angular and energy
components. The ingredients are standard higher-order tools adapted
here to the coupled proton transport system. For the energy dynamics,
we apply a Milstein discretisation to the log-transformed variable
$Y_t=\log\qp{E_t}$, ensuring positivity while achieving strong $O(h)$
convergence. The angular integrator is then based on the
Runge--Kutta--Munthe--Kaas (RKMK) framework, which lifts the dynamics
into the Lie algebra, performs a higher-order update and maps back to
the sphere via the exponential map
\cite{muniz2022higher,iserles2000lie}. These two components are
combined into a fully structure-preserving scheme for the coupled
system.

\subsection{Log-Energy Milstein Scheme}

The evolution of energy follows the stochastic differential equation
\begin{equation}
  \label{eq:energyhigh}
  \d E_t = -S(E_t) \d t + \sqrt{T(E_t)} \d W_t^E.
\end{equation}
To improve numerical accuracy while preserving the positivity of
energy, we apply a second-order Milstein discretisation to the
logarithmic variable $Y_t = \log\qp{E_t}$. This approach avoids the
need for artificial projection and ensures that the reconstructed
energy $E_t = \exp\qp{Y_t}$ remains strictly positive.

Applying Itô's lemma, the log-transformed SDE takes the form
\begin{equation}
  \d Y_t
  =
  \qp{
    -S(\exp\qp{Y_t}) \exp\qp{-Y_t}
    - \tfrac{1}{2} T(\exp\qp{Y_t}) \exp\qp{-2Y_t}
  } \d t
  +
  \sqrt{T(\exp\qp{Y_t})} \exp\qp{-Y_t} \d W_t^E.
\end{equation}
Applying the Milstein scheme to this SDE yields the update
\begin{equation}
  \label{eq:milstein-log}
  \begin{split}
    Y_{n+1}
    &=
    Y_n
    -
    h S(\exp\qp{Y_n}) \exp\qp{-Y_n}
    -
    \tfrac{h}{2} T(\exp\qp{Y_n}) \exp\qp{-2Y_n}
    +
    \sqrt{T(\exp\qp{Y_n})} \exp\qp{-Y_n}  \xi_n^E
    \\
    &\qquad + \frac{1}{2} \qp{ \frac{T'(\exp\qp{Y_n})\exp\qp{-Y_n}}{2} - T(\exp\qp{Y_n})\exp\qp{-2Y_n}} \qp{ (\xi_n^E)^2 - h },
  \end{split}
\end{equation}
where $\xi_n^E \sim \mathcal{N}(0, h)$.  

\begin{lemma}[Positivity and strong order $1$ for the log--energy Milstein scheme]
\label{lem:milstein-log-energy}
Let $E_t$ solve \eqref{eq:energyhigh} on $[0,T]$ with initial
$E_0>E_{\min}>0$ and impose killing at $E_{\min}$.  Assume $S,T\in
C^2([E_{\min},\infty))$, $T(E)\ge c_0>0$ on $[E_{\min},\infty)$.  Let
    $Y_t=\log E_t$ and let $Y_{n}$ be defined by the Milstein update
    \eqref{eq:milstein-log} with stepsize $h=T/N$, and set
    $E_n=\exp(Y_n)$.  Then, $E_n>0$ for all $n$, and there exists a
    constant $C$ (independent of $h$) such that
    \begin{equation}
      \mathbb{E}\qb{\sup_{0\le n\le N} \norm{E_n - E(t_n)}^2} \le C h^2 .
    \end{equation}    
    In particular, $E_n$ converges to $E(t_n)$ with strong order $1$.
\end{lemma}

\begin{proof}
  
Positivity is immediate since $E_n=\exp(Y_n)$. Write the log-energy
SDE as
\begin{equation}
  \begin{split}
    \d Y_t
    &=
    \mu(Y_t) \d t
    +
    b(Y_t) \d W_t^E,
    \\
    \mu(y)
    &=
    -S(\exp\qp y)\exp\qp{-y}
    -
    \tfrac{1}{2} T(\exp\qp y)\exp\qp{-2y}
    \\
    b(y) &= \sqrt{T(\exp\qp y)} \exp\qp{-y}.
  \end{split}
\end{equation}
Because $E_t\in[E_{\min},\infty)$ up to the killing time and $S,T\in
  C^2([E_{\min},\infty))$ with bounded derivatives, the compositions
    $\mu,b\in C_b^2(\mathbb{R})$ (globally Lipschitz with bounded
    first and second derivatives).  Hence the classical Milstein
    convergence theorem (see, e.g.,
    \cite[Thm.~10.6.2]{kloeden1992stochastic}) applies to the scheme
    \eqref{eq:milstein-log}, giving
    \begin{equation}
      \mathbb{E}\qb{\sup_{0\le n\le N} \norm{Y_n - Y(t_n)}^2} \le C h^2 .
    \end{equation}
    To transfer this estimate to $E=\exp(Y)$, use the mean value
    theorem. For each $n$ there exists $\zeta_n$ between $Y_n$ and
    $Y(t_n)$ such that
    \begin{equation}
      \norm{E_n - E(t_n)} 
      =
      \norm{\exp\qp{Y_n} - \exp\qp{Y(t_n)}}
      =
      \exp\qp{\zeta_n}\norm{Y_n - Y(t_n)}.
    \end{equation}    
    Since $E(t)\ge E_{\min}$ up to killing and $E_n=\exp(Y_n)>0$, we
    have $\exp\qp{\xi_n}\le C$ uniformly on $[0,T]$ (the constant depends
    on $E_{\min}$ and on a moment bound for $E$ but not on
    $h$). Squaring and taking the supremum over $n$,
    \begin{equation}
      \sup_{0\le n\le N} \norm{E_n - E(t_n)}^2
      \le
      C \sup_{0\le n\le N} \norm{Y_n - Y(t_n)}^2 .
    \end{equation}
    Taking expectations and combining with the Milstein bound for $Y$
    yields the desired result.
\end{proof}

\subsection{Lie-Group Angular Integrator}

Let $\{U_\alpha\}_{\alpha=1}^M$ be a finite smooth atlas of
$\mathbb{S}^2$. On each chart $U_\alpha$, choose a smooth oriented
orthonormal frame $\{u_{\alpha,1},u_{\alpha,2}\}$ of the tangent
bundle. The local vector fields are then defined by
\begin{equation}
  \sigma_{\alpha,i}(\Omega)
  :=
  u_{\alpha,i}(\Omega)\times\Omega,
  \qquad i=1,2.
\end{equation}
The numerical update is implemented extrinsically, so no global frame
is required in the algorithm itself; the local frames are used only
for the local representation and analysis.

On a fixed chart $U_\alpha$, we write $\sigma_i=\sigma_{\alpha,i}$ and
$u_i=u_{\alpha,i}$ for brevity. The angular dynamics then take the
local Stratonovich form
\begin{equation}
  \d\Omega_t
  =
  \sqrt{2\epsilon(E_t)}
  \qp{ \sigma_1(\Omega_t)\circ\d W_t^{(1)}
    + \sigma_2(\Omega_t)\circ\d W_t^{(2)}},
\end{equation}
where $W^{(1)},W^{(2)}$ are independent standard Brownian motions.
The fields $\sigma_1,\sigma_2$ are smooth, tangent and satisfy the
local commutator identity
\begin{equation}
  [\sigma_1,\sigma_2](\Omega)= - \Omega,
\end{equation}
so they do not commute. This bracket (a rotation about the normal
axis) is the source of the L\'evy--area correction in strong
order--$1$ integrators.

Writing $\Omega=(\sin\theta\cos\phi,\sin\theta\sin\phi,\cos\theta)$, a
standard orthonormal frame (smooth away from the poles) is
\begin{equation}
  u_1(\Omega)=(\cos\theta\cos\phi, \cos\theta\sin\phi, -\sin\theta),
  \qquad
  u_2(\Omega)=(-\sin\phi,\cos\phi,0).
\end{equation}

Over one step $[t_n,t_{n+1}]$ we first approximate the diffusion clock increment
\begin{equation}
  \label{eq:diff-clock}
  \Delta\gamma_n
  :=
  h\big(\epsilon(E_n)+\epsilon(E_{n+1})\big)
  \approx
  2\int_{t_n}^{t_{n+1}} \epsilon(E_s) \d s
  =:
  \Delta\gamma_n^{\mathrm{ex}}
\end{equation}
which acts as the process's internal time. 

Now draw
\begin{equation}
  (\Delta W_n^{(1)},\Delta W_n^{(2)})
  \sim
  \mathcal N(0,\Delta\gamma_n I_2)
\end{equation}
together with the scalar L\'evy area
\begin{equation}
  A_n
  :=
  \frac12\int_{t_n}^{t_{n+1}}
  \big(\d W_s^{(1)} \d W_s^{(2)}-\d W_s^{(2)} \d W_s^{(1)}\big),
\end{equation}
sampled jointly with the Brownian increments using the Wiktorsson
construction for two-dimensional Brownian motion. In the present
two-dimensional driving setting, only a single scalar L\'evy-area
variable is required at each step. Thus the higher-order angular
update introduces only constant additional work per timestep relative
to the geometric Euler method. We note that $A_n$ has mean $0$ and
variance $\Delta\gamma_n^2/12$.

Define the Lie algebra increment
\begin{equation}
  \label{eq:algebra-el}
  \xi_n
  :=
  \Delta W_n^{(1)} u_1(\Omega_n)
  + \Delta W_n^{(2)} u_2(\Omega_n)
  - A_n\Omega_n.
\end{equation}

Finally, perform the angular update as a rotation
\begin{equation}
  \label{eq:omega-update}
  \Omega_{n+1} = \Exp{\Omega_n}(\xi_n).
\end{equation}

\begin{lemma}[Angular strong order $1$]
\label{lem:ang-order1}
Assume the conditions of Lemma \ref{lem:milstein-log-energy} hold.
Suppose the diffusion clock (\ref{eq:diff-clock}) is approximated by
an adapted quadrature based on the Milstein energy such that
\begin{equation}
  \mathbb{E}\qb{\norm{\Delta\gamma_n-\Delta\gamma_n^{\mathrm{ex}}}^2}
  =
  O(h^3).
\end{equation}
With the Lie Algebra element defined by (\ref{eq:algebra-el}) and the
update (\ref{eq:omega-update}), there exists $C>0$ independent of $h$
such that
\begin{equation}
  \mathbb{E}\qb{\sup_{0\le n\le N}\norm{\Omega_n-\Omega(t_n)}^2} \le C h^2.
\end{equation}
\end{lemma}

\begin{proof}
  Define the clock $\gamma(t)=2\int_0^t\epsilon(E_s)\d s$ and let
  $s_n=\gamma(t_n)$, so
  $\Delta\gamma_n^{\mathrm{ex}}=s_{n+1}-s_n$. Consider the
  time-changed process $\widetilde\Omega_s:=\Omega_{\gamma^{-1}(s)}$.
  By the time-change theorem for Stratonovich SDEs, $\widetilde\Omega$
  solves on $[s_n,s_{n+1}]$
  \begin{equation}
    \d\widetilde\Omega_s
    =
    \sigma_1(\widetilde\Omega_s)\circ \d B_s^{(1)}
    +
    \sigma_2(\widetilde\Omega_s)\circ \d B_s^{(2)},
  \end{equation}  
  where $B^{(1)},B^{(2)}$ are independent Brownian motions in the
  clock $s$. Here $\sigma_1$ and $\sigma_2$ denote the local vector
  fields on a fixed chart, with the chart index suppressed for
  readability. All Taylor and Lipschitz estimates are carried out in
  local charts. Compactness of $\mathbb{S}^2$ and a finite atlas give
  uniform constants. Let $\Delta
  B_n^{(i)}=B_{s_{n+1}}^{(i)}-B_{s_n}^{(i)}$ and
  $A_n^{\mathrm{ex}}=\frac12\int_{s_n}^{s_{n+1}}\big(\d B_s^{(1)}\d
  B_s^{(2)}-\d B_s^{(2)}\d B_s^{(1)}\big)$. On each local chart, the
  time-changed equation is driven by smooth bounded local vector
  fields, so the Stratonovich stochastic Taylor expansion applies
  locally. The resulting constants are uniform over the finite atlas
  yielding
  \begin{equation}
    \widetilde\Omega_{s_{n+1}}
    =
    \widetilde\Omega_{s_n}
    + \sum_{i=1}^2 \sigma_i(\widetilde\Omega_{s_n}) \Delta B_n^{(i)}
    + [\sigma_1,\sigma_2](\widetilde\Omega_{s_n}) A_n^{\mathrm{ex}}
    + R_n^{\mathrm{ex}},
    \label{eq:ST}
  \end{equation}
  with a remainder satisfying
  \begin{equation}
    \mathbb{E}\qb{\norm{R_n^{\mathrm{ex}}}^2}
    \le
    C (\Delta\gamma_n^{\mathrm{ex}})^3.
  \end{equation}  
  Here all coefficients and their derivatives are bounded because
  $\mathbb{S}^2$ is compact and terms of weight $\ge 3/2$ in the
  Stratonovich hierarchy contribute
  $O\big((\Delta\gamma_n^{\mathrm{ex}})^{3/2}\big)$ in RMS, hence
  $O\big((\Delta\gamma_n^{\mathrm{ex}})^3\big)$ in mean
  square. Moreover, with the oriented frame $u_1\times u_2=\Omega$ one
  has $[\sigma_1,\sigma_2](\Omega)=-\Omega$.

  Now, freeze the frame at $\Omega_n$ and form
  \begin{equation}
    \xi_n=\Delta W_n^{(1)}u_1(\Omega_n)+\Delta W_n^{(2)}u_2(\Omega_n)-A_n \Omega_n,
  \end{equation}
  with $(\Delta W_n^{(1)},\Delta W_n^{(2)},A_n)$ sampled over duration
  $\Delta\gamma_n$.  For small $\xi$, the Riemannian exponential on
  $\mathbb{S}^2$ admits the expansion
  \begin{equation}
    \Exp{\Omega_n}(\xi) = \Omega_n + \sigma_1(\Omega_n)\Delta W_n^{(1)}
    + \sigma_2(\Omega_n)\Delta W_n^{(2)}
    + [\sigma_1,\sigma_2](\Omega_n)A_n
    + \widehat R_n,
  \end{equation}
  where $\mathbb{E}\qb{\norm{\widehat R_n}^2}\le C
  (\Delta\gamma_n)^3$.  This follows by writing the update as the
  rotation $\exp([\xi_n]_\times)\Omega_n$ (with $[\cdot]_\times$ the
  skew matrix), expanding the flow map in the Lie algebra via the
  Baker--Campbell--Hausdorff series and using
  $\norm{\xi_n}=O\qp{(\Delta\gamma_n)^{1/2}}$ in RMS and the bracket
  identity $[\sigma_1,\sigma_2](\Omega_n)=-\Omega_n$. The cubic bound
  on $\widehat R_n$ in mean square is a consequence of the smoothness
  of the exponential map on the compact manifold and the boundedness
  of the fields and their derivatives.

  We now compare the exact clock increments $(\Delta
  B_n^{(i)},A_n^{\mathrm{ex}})$ over $\Delta\gamma_n^{\mathrm{ex}}$
  with the simulated increments $(\Delta W_n^{(i)},A_n)$ over
  $\Delta\gamma_n$. By construction
  \begin{equation}
    \mathbb{E}\qb{\norm{\Delta W_n-\Delta B_n}^2}
    =
    \mathbb{E}\qb{\norm{\Delta\gamma_n-\Delta\gamma_n^{\mathrm{ex}}}^2}=O(h^{3/2})
  \end{equation}
  in RMS, hence $O(h^3)$ in mean square. Similarly,
  $\mathbb{E}\qb{\norm{A_n-A_n^{\mathrm{ex}}}^2}=O(h^3)$ because
  $\operatorname{Var}(A)=\Delta\gamma^2/12$ and the variance error is
  $O(h^3)$. Since the coefficients $\sigma_i$ and the bracket are
  bounded, replacing $(\Delta B_n^{(i)},A_n^{\mathrm{ex}})$ by
  $(\Delta W_n^{(i)},A_n)$ in the expansions \eqref{eq:ST} and in the
  numerical update produces a mean-square perturbation of order
  $O(h^3)$.

  Combining the above and identifying
  $\widetilde\Omega_{s_n}=\Omega(t_n)$, the one-step error in clock
  time satisfies
  \begin{equation}
    \mathbb{E}\qb{
      \norm{\Omega_{n+1}-\Omega(t_{n+1})}^2
      \big|  \mathcal F_{t_n}}
    \le (1+C h)\|\Omega_n-\Omega(t_n)\|^2 + C h^3,
  \end{equation}
since $\Delta\gamma_n^{\mathrm{ex}}=O(h)$ and all Lipschitz constants are uniform
(on the compact manifold). Taking expectations and iterating yields
\begin{equation}
  \mathbb{E}\qb{\norm{\Omega_{n+1}-\Omega(t_{n+1})}^2
  }
  \le
  (1+Ch) \mathbb{E}\qb{\norm{\Omega_n-\Omega(t_n)}^2} + C h^3.
\end{equation}
By the discrete Gr\"onwall lemma this gives
$\mathbb{E}\qb{\norm{\Omega_n-\Omega(t_n)}^2 } \le C h^2$ uniformly in
$n$.  A standard Doob-Kolmogorov argument then yields the desired result.
\end{proof}

\begin{theorem}[Strong convergence of the full scheme]
\label{thm:strong-rkmk}
Suppose the conditions of Lemmata \ref{lem:milstein-log-energy} and
\ref{lem:ang-order1} hold. Let the tuple $(X_t,Y_t,\Omega_t)$ solve
the SDE system (\ref{eq:LogSDE}) on $[0,T]$, with killing at
$E_{\min}>0$. Let $(X_n,Y_n,\Omega_n)_n$ be the numerical approximation
given by
\begin{equation}
  \label{eq:LogMilsteinExp}   
  \begin{split}
    X_{n+1} &= X_n + h  \Omega_n,
    \\
    Y_{n+1}
    &=
    Y_n
    -
    h S(\exp\qp{Y_n}) \exp\qp{-Y_n}
    -
    \tfrac{h}{2} T(\exp\qp{Y_n}) \exp\qp{-2Y_n}
    +
    \sqrt{T(\exp\qp{Y_n})} \exp\qp{-Y_n}  \xi_n^E
    \\
    &\qquad + \frac{1}{2} \qp{ \frac{T'(\exp\qp{Y_n})\exp\qp{-Y_n}}{2} - T(\exp\qp{Y_n})\exp\qp{-2Y_n}} \qp{ (\xi_n^E)^2 - h },
    \\
    \Omega_{n+1} &= \Exp{\Omega_n}(\xi_n),
  \end{split}
\end{equation}
with $\xi_n$ prescribed by (\ref{eq:algebra-el}) and $E_n = \exp{Y_n}$.

Then, there exists $C>0$ independent of $h$ such that
\begin{equation}
  \mathbb{E}
  \qb{
    \sup_{0\le n\le N}\qp{
      \norm{X_n-X(t_n)}^2
      +
      \norm{E_n-E(t_n)}^2
      +
      \norm{\Omega_n-\Omega(t_n)}^2 }
  }
  \le C h^2.
\end{equation}
\end{theorem}

\begin{proof}
  By Lemma~\ref{lem:milstein-log-energy} we have
  \begin{equation}
    \mathbb{E}\qb{\sup_{0\le n\le N}
      \norm{E_n-E(t_n)}^2
    }\le C h^2,
  \end{equation}  
  and by Lemma~\ref{lem:ang-order1},
  \begin{equation}
    \mathbb{E}\qb{\sup_{0\le n\le N}\norm{\Omega_n-\Omega(t_n)}^2} \le C h^2.
  \end{equation}  
  It remains to bound the spatial error. Write
  \begin{equation}
    X_{n+1}-X(t_{n+1})
    =
    \qp{
      X_n-X(t_n)}
    + h\qp{\Omega_n-\Omega(t_n)}
    + \int_{t_n}^{t_{n+1}}\qp{\Omega(t_n)-\Omega(s)} \d s .
  \end{equation}
Taking norms, squaring and using $(a+b+c)^2\le 3(a^2+b^2+c^2)$,
\begin{equation}
  \norm{X_{n+1}-X(t_{n+1})}^2
  \le
  3\norm{X_n-X(t_n)}^2
  +
  3h^2\norm{\Omega_n-\Omega(t_n)}^2
  +
  3\norm{\int_{t_n}^{t_{n+1}}(\Omega(t_n)-\Omega(s)) \d s }^2.
\end{equation}
For the last term, by Cauchy–Schwarz,
\begin{equation}
  \norm{\int_{t_n}^{t_{n+1}}(\Omega(t_n)-\Omega(s))\d s}^2
  \le
  h \int_{t_n}^{t_{n+1}}\norm{\Omega(t_n)-\Omega(s)}^2 \d s .
\end{equation}
Since $\Omega$ is a diffusion with smooth coefficients on the compact
manifold $\mathbb{S}^2$, $\mathbb{E}\qb{\norm{\Omega(t_n)-\Omega(s)}}^2
  \le C \norm{s-t_n}$.  Hence
  \begin{equation}
    \mathbb{E}\qb{\norm{\int_{t_n}^{t_{n+1}}(\Omega(t_n)-\Omega(s)) \d s }^2}
    \le
    C h \int_{t_n}^{t_{n+1}} (s-t_n) \d s
    = C h^3 .
  \end{equation}  
Taking expectations in the recursive bound and summing, we get
\begin{equation}
  \mathbb{E}\qb{\norm{X_{n+1}-X(t_{n+1})}^2}
  \le
  (1+Ch) \mathbb{E}\qb{\norm{X_n-X(t_n)}^2}
  + C h^2 \mathbb{E}\qb{\norm{\Omega_n-\Omega(t_n)}^2}
  + C h^3 .
\end{equation}
Using $\mathbb{E}\qb{\sup_n\norm{\Omega_n-\Omega(t_n)}^2} \le C h^2$ from
Lemma~\ref{lem:ang-order1} and applying the discrete Gr\"onwall lemma
yields
\begin{equation}
  \mathbb{E}\qb{\sup_{0\le n\le N}\norm{X_n-X(t_n)}^2} \le C h^2 .
\end{equation}
Combining the three componentwise bounds proves the stated estimate.
\end{proof}

\begin{remark}[Structure preservation under the higher–order scheme]
  The structure preserving features of the geometric Euler method are
  retained by the higher accuracy scheme. In energy, the Milstein
  update for the log variable preserves positivity exactly and
  improves pathwise accuracy to strong order $1$ without any
  projection. In angle, the update is performed intrinsically on
  $\mathbb{S}^2$ via the exponential map, so the unit-norm constraint
  is preserved exactly at every step. For the angular accuracy
  incorporating the L\'evy-area terms yield a Milstein/Castell-Gaines
  integrator of strong order $1$.

  These refinements do not alter the geometric or physical constraints
  of the model, but they do improve convergence rates while
  maintaining fidelity to the manifold structure.  The full scheme
  therefore remains confined to the physically admissible state space
  and is well suited to finite-time pathwise simulation and to the
  sensitivity computations developed later.
\end{remark}

\subsection{Numerical Examples}

To illustrate the improved behaviour of the higher-order scheme, we
repeat the angular diffusion experiment from the previous section, now
using the second-order RKMK method. The update remains constrained to
the unit sphere but achieves higher pathwise accuracy and better
approximation of finite-time pathwise trajectories and of empirical
angular statistics. Figure~\ref{fig:angle_drift} shows the angular
trajectory and empirical histogram for $\Omega_n \in \mathbb{S}^1$
over a long simulation. Compared to the first-order geometric method,
the second-order update yields comparable geometric fidelity while
providing reduced bias in the angular statistics and improved
convergence in observable quantities.

\begin{figure}[h!]
  \centering
  \captionsetup[subfigure]{justification=centering}
  \begin{subfigure}[b]{0.45\textwidth}
    \includegraphics[width=\textwidth]{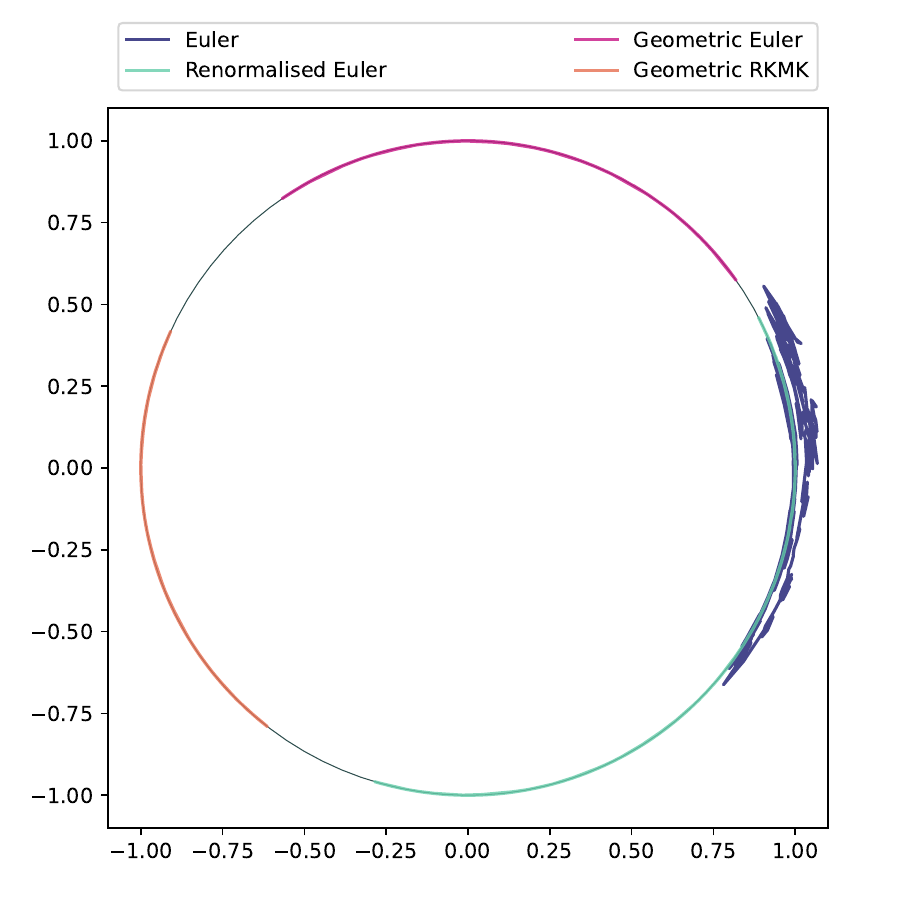}
    \caption{Trajectory on the circle over time.\newline $h=0.01,T=5,\epsilon_0=0.1$}
    \label{fig:circle_paths}
  \end{subfigure}
  \hfill
  \begin{subfigure}[b]{0.45\textwidth}
    \includegraphics[width=\textwidth]{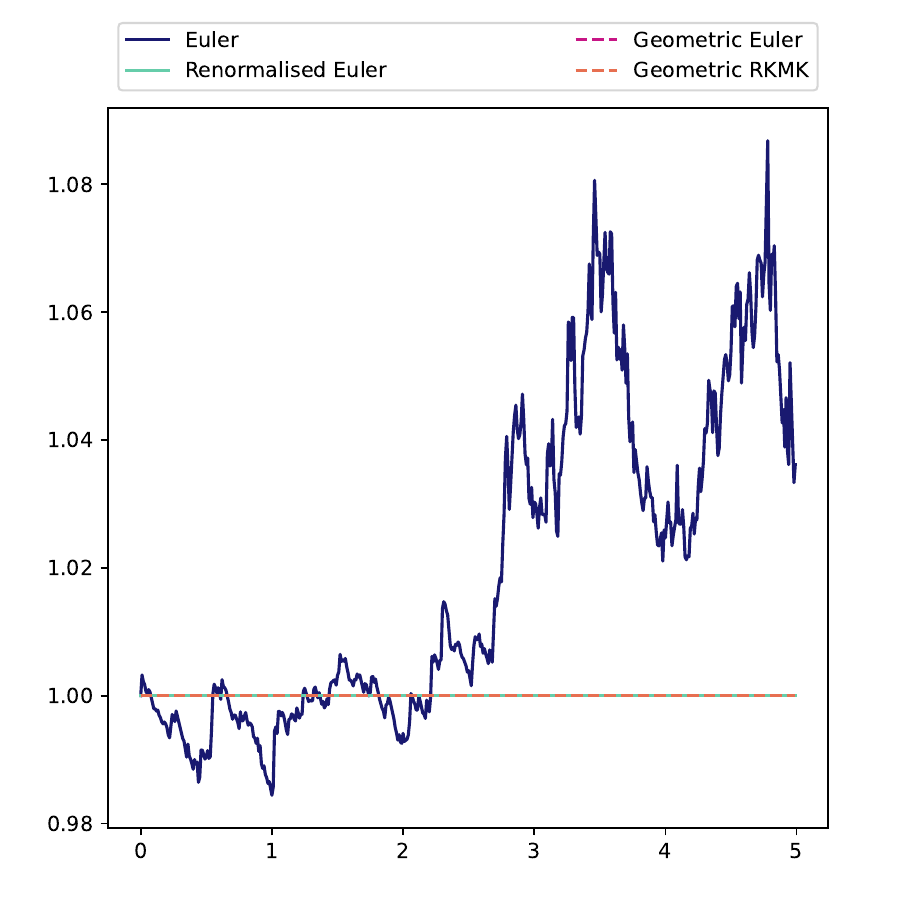}
    \caption{Norm of direction vector $\Omega_n$ over time.\newline $h=0.01,T=5,\epsilon_0=0.1$}
    \label{fig:circle_norms}
  \end{subfigure}
  \hfill
  \begin{subfigure}[b]{0.45\textwidth}
    \includegraphics[width=\textwidth]{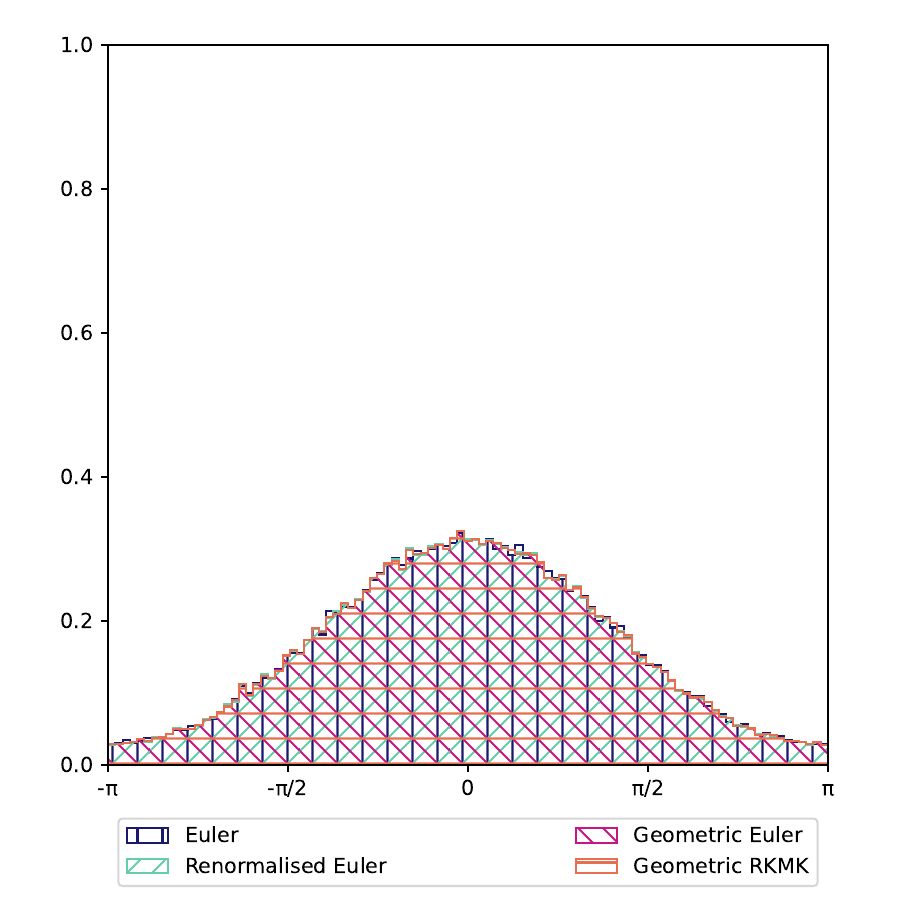}
    \caption{Histogram of angles $\arg(\Omega_n)$ at final time.\newline $h=0.01,T=8,\epsilon_0=0.1$}
    \label{fig:circle_hists}
  \end{subfigure}
  \hfill
  \begin{subfigure}[b]{0.45\textwidth}
    \includegraphics[width=\textwidth]{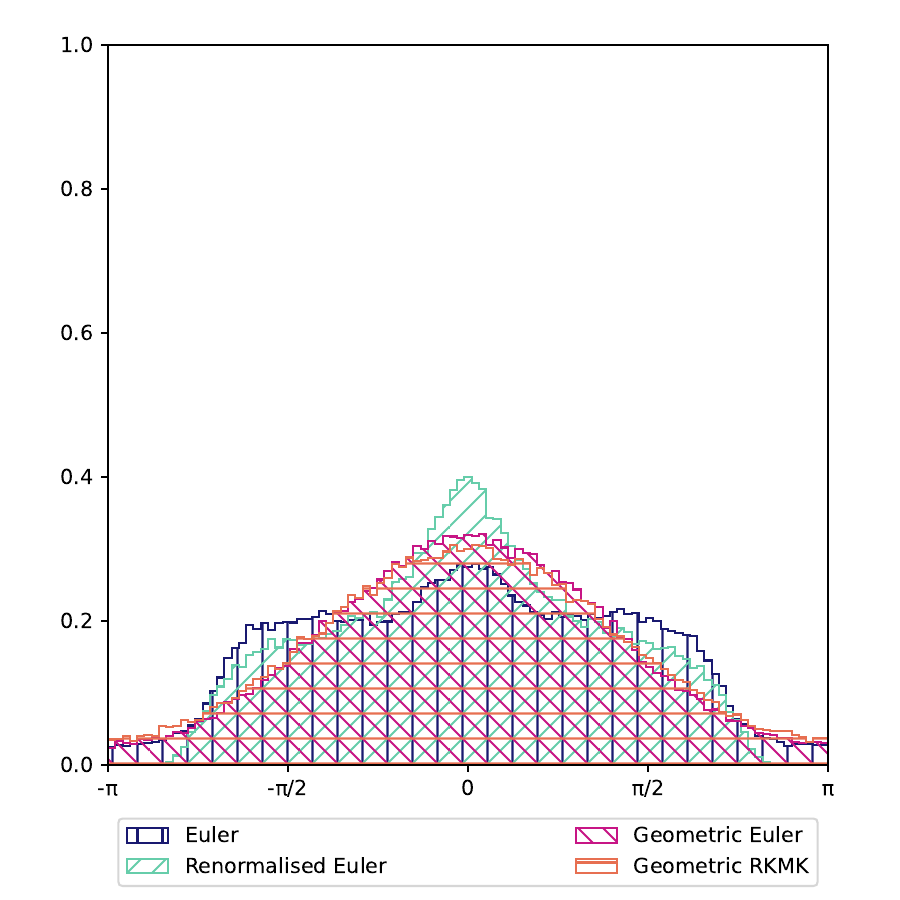}
    \caption{Histogram of angles $\arg(\Omega_n)$ at final time.\newline $h=4,T=8,\epsilon_0=0.1$}
    \label{fig:circle_hists_2}
  \end{subfigure}

  \caption{Angular diffusion on $\mathbb{S}^1$.  Comparison of naive
    Euler, renormalised Euler, first-order geometric and higher-order
    RKMK updates. The naive Euler schemes exhibit norm drift and bias
    in the angular distribution, while the exponential-map integrators
    (geometric Euler and higher-order RKMK) preserve the geometry.
    The higher-order scheme additionally reduces statistical bias and
    retains the correct distributional properties even for large timesteps.}
  \label{fig:angle_drift}
\end{figure}

\begin{example}[Strong convergence]\label{ex:strong-convergence}
  We illustrate the strong convergence rates of the numerical schemes
  used to approximate the SDE systems \autoref{eq:SDE1} and
  \autoref{eq:SDE2}, following the constructions in
  \autoref{sec:euler} and \autoref{sec:high-order}. Model parameters
  are fixed as $\alpha=0.022$, $p=1.77$, $\kappa=0.075$,
  $\epsilon_0=10^{-5}$.

  The strong error is computed at final time $T=0.1$, relative to a
  reference solution obtained with step size $h=10^{-4}$ and averaged
  over $1000$ independent sample paths.  The results are shown in
  \autoref{fig:strong-convergence}. For the angular component,
  \autoref{fig:convergence1} shows the Geometric Euler scheme, which
  attains strong order $1/2$, while \autoref{fig:convergence2} shows
  the Geometric RKMK scheme with L\'evy areas, which attains strong
  order $1$. For the energy component, the Euler-Maruyama and Milstein
  schemes (together with their log-transformed variants) achieve the
  theoretically expected rates of order $1/2$ and $1$. For the spatial
  component, the deterministic Euler update is sufficient to obtain
  order $1$, since the $X_t$-equation contains no diffusion term.
\end{example}

\begin{figure}[h!]
  \centering
  \captionsetup[subfigure]{justification=centering}
  \begin{subfigure}[b]{0.4\textwidth}
    \includegraphics[width=\textwidth]{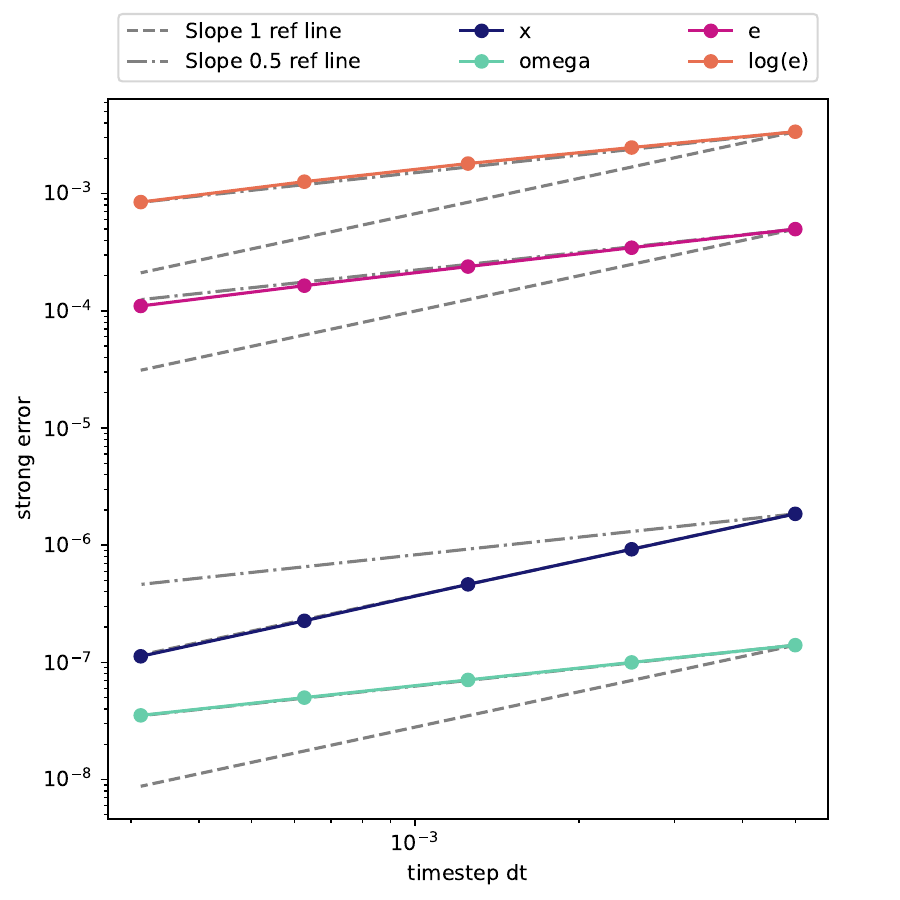}
    \caption{Geometric Euler: order $1/2$ convergence in angle.}
    \label{fig:convergence1}
  \end{subfigure}
  \hfill
  \begin{subfigure}[b]{0.4\textwidth}
    \includegraphics[width=\textwidth]{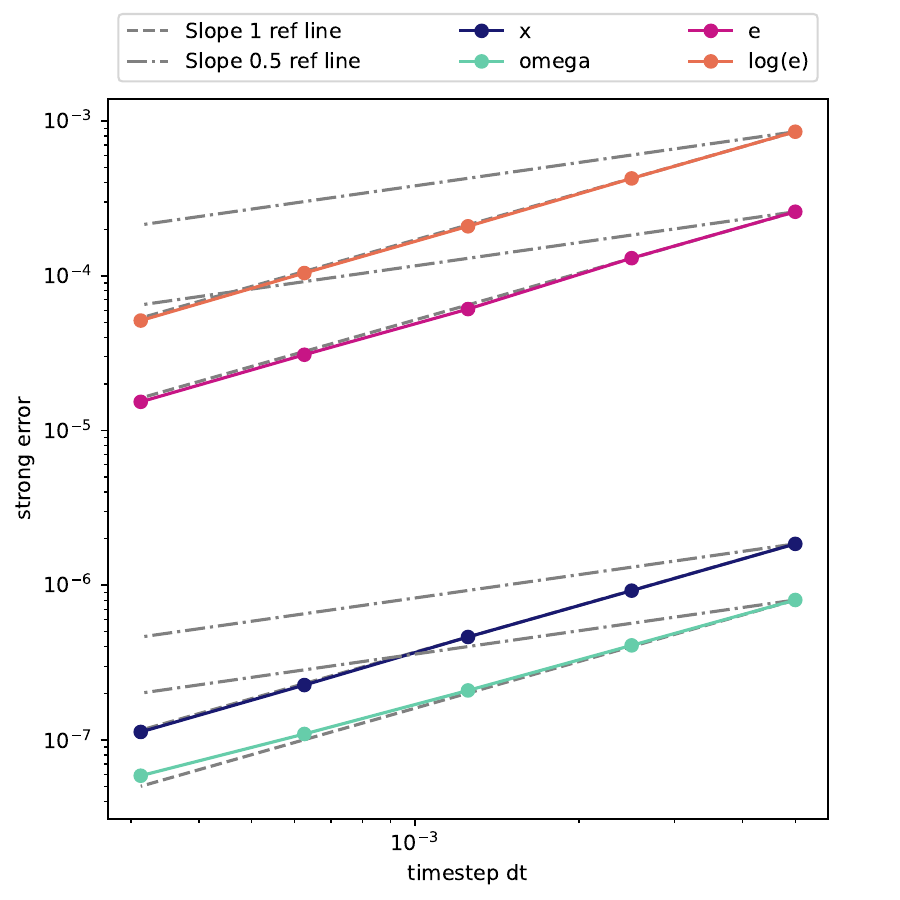}
    \caption{Geometric RKMK: order $1$ convergence in angle.}
    \label{fig:convergence2}
  \end{subfigure}
  \caption{Example~\ref{ex:strong-convergence}. Strong convergence of
    numerical schemes.  Energy schemes achieve their expected rates
    ($1/2$ for Euler-Maruyama, $1$ for Milstein), angular schemes
    achieve $1/2$ (Geometric Euler) and $1$ (Geometric RKMK with
    L\'evy areas) and the spatial component achieves order $1$ due to
    its deterministic structure.}
  \label{fig:strong-convergence}
\end{figure}

\section{Sensitivity of quantities of interest}
\label{sec:dose-sensitivity}

The energy dynamics in \eqref{eq:LogSDE} depend on the parameters
$\alpha$, $p$ and $\kappa$, which enter through the Bragg--Kleeman
stopping power law and the straggling variance. In practice, these
parameters are often uncertain, require calibration, or must be
analysed for their influence on physically relevant observables. It is
therefore natural to study the sensitivity of such quantities of
interest with respect to variations in $\alpha$, $p$ and $\kappa$.

\subsection{Introducing quantities of interest}

In proton transport, am important observable is the expected dose deposited,
defined by
\begin{align}
  \label{eq:dose}
  D(x;\theta)
  :=
  \mathbb{E}\qb{ \int_0^{T \wedge \tau(\theta)} S(E_t)\varphi(X_t-x)\d t },
\end{align}
where $\varphi$ is a spatial kernel that localises energy deposition
near $x$ and $S(E_t)$ is the stopping power. This functional depends
on parameters $\theta \in \{\alpha, p, \kappa\}$ through the energy
path $E_t = \exp(Y_t)$. The stopping time $\tau(\theta) := \inf \{ t
\geq 0 : E_t \leq E_{\min} \}$ also depends on $\theta$ via $E_t$. As
before, we impose a killing boundary at $E_{\min}>0$ so that the
process is terminated when the energy drops below this threshold.

We are interested in the sensitivity of $D(x;\theta)$ with respect to
$\theta$. In the deterministic energy evolution case when $\kappa=0$, this derivative may be
moved inside the expectation because the stopping time is an explicit smooth
function of the parameters. In the stochastic case $\kappa\neq 0$,
however, the dependence of the hitting time $\tau(\theta)$ on $\theta$
is not sufficiently regular for a pathwise differentiation argument
based on the stopped functional $D(x;\theta)$. Accordingly, we
distinguish between these two settings below. The stopped dose
functional is differentiated directly only in the deterministic
regime, while for $\kappa\neq 0$ we introduce in
\autoref{sec:straggling_sens} a regularised surrogate observable whose
pathwise sensitivity is well defined and can be approximated
numerically.

\begin{remark}[Scope of the sensitivity analysis]
  The difficulty in the stochastic case stems from the dependence of
  the stopping time
  \begin{equation}
    \tau(\theta)=\inf\{t\ge 0: E_t\le E_{\min}\}
  \end{equation}
  on the parameter $\theta$. When $\kappa=0$, the energy path is
  deterministic and $\tau(\theta)$ may be differentiated explicitly
  for $\theta\in\{\alpha,p\}$. When $\kappa\neq 0$, small parameter
  perturbations may produce order-one changes in the hitting time, so
  the functional inside the expectation lacks 
  sufficient regularity for the exchange of the order of limits between the 
  derivative and the expectation.
  Hence the stopped functional $D(x;\theta)$ is not suitable for direct
  pathwise differentiation. 
  For this reason, the stochastic
  sensitivity analysis developed later is carried out for a
  regularised dose-type observable rather than for the stopped
  functional itself.
\end{remark}

\subsection{Deterministic case ($\kappa=0$)}\label{sec:no_straggling_sens}

When $\kappa=0$, the energy evolution becomes deterministic and satisfies the ODE
\begin{equation}
    \frac{\d E}{\d t} = -S(E), \qquad E(0)=E_0.
\end{equation}
With $S(E)=\qp{p \alpha}^{-1}E^{1-p}$, the closed-form solution is
\begin{equation}
    E(t) = \bigl(E_0^p - t/\alpha\bigr)^{1/p},
\end{equation}
which is strictly decreasing and well defined up to the critical time
$T_c = \alpha E_0^p$ at which $E(t)$ vanishes.

The deterministic stopping time $T(\theta)$, defined by
$E(T(\theta))=E_{\min}$, is therefore
\begin{equation}
    T(\theta) = \alpha \bigl(E_0^p - E_{\min}^p\bigr).
\end{equation}
The dose functional reduces to
\begin{equation}\label{eq:no-straggling-dose}
    D(x;\theta) = \mathbb{E} \left[ \int_0^{T(\theta)} S(E_t) \varphi(X_t-x) \d t \right],
\end{equation}
where randomness remains only through the spatial trajectory $X_t$ and
angular component $\Omega_t$.

Differentiating with respect to $\theta \in \{\alpha,p\}$ gives
\begin{equation}
  \label{eq:dose-sens-no-straggling}
  \begin{split}
    \frac{\d}{\d\theta} D(x;\theta)
    &=
    \mathbb{E}\Bigg[
      S(E_{T(\theta)}) \varphi(X_{T(\theta)}-x) \frac{\d T}{\d\theta}
      \\
      &\qquad +
      \int_0^{T(\theta)}
      \qp{
        \frac{\partial S}{\partial \theta}(E_t)
        +
        \frac{\partial S}{\partial E}(E_t) \partial_\theta E_t}
      \varphi(X_t-x)\d t \Bigg],
  \end{split}
\end{equation}
where $E_{T(\theta)}=E_{\min}$ by construction. The stopping-time
sensitivities are
\begin{equation}
  \begin{split}
    \frac{\d T}{\d \alpha}
    &=
    E_0^p - E_{\min}^p,
    \\
    \frac{\d T}{\d p}
    &=
    \alpha \qp{\log(E_0)E_0^{p} - \log(E_{\min})E_{\min}^{p}}.
  \end{split}
\end{equation}
Since $E(t)$ is known in closed form, its parameter derivatives can also be computed explicitly
\begin{equation}
  \begin{split}
    \partial_{\alpha} E(t)
    &=
    \frac{t}{\alpha^2 p} \qp{E_0^p - t/\alpha}^{1/p - 1},
    \\
    \partial_{p} E(t)
    &=
    \frac{E_0^p \log(E_0)}{p} \qp{E_0^p - t/\alpha}^{1/p-1}
    -
    \frac{1}{p^2}\log \qp{E_0^p - \tfrac{t}{\alpha}} \qp{ E_0^p - t/\alpha }^{1/p}.
  \end{split}
\end{equation}

\subsection{Regularisation under stochastic straggling}
\label{sec:straggling_sens}

When $\kappa\neq 0$, the energy path is stochastic and the hitting
time
\begin{equation}
  \tau(\theta)=\inf\{t\ge 0:E_t\le E_{\min}\}
\end{equation}
is not sufficiently regular in the parameter $\theta$ for a direct
pathwise differentiation of the stopped functional
$D(x;\theta)$. We therefore introduce a regularised surrogate
observable and carry out the sensitivity analysis for that quantity.

We first define the indicator-based surrogate
\begin{equation}
  D^{\mathrm{ind}}(x;\theta)
  :=
  \mathbb{E}
  \qb{
    \int_0^T
    S(E_t) \varphi(X_t-x) \mathbf{1}_{\{E_t>E_{\min}\}}
    \d t
  }.
\end{equation}
The discontinuity of the indicator still obstructs pathwise
differentiation, so we replace it by a smooth transition function
\begin{equation}
  \label{eq:mollified-indicator}
  \mathcal{I}_\delta(E,E_{\min})
  :=
  \tfrac12\qp{
    1+\tanh\qp{\tfrac{E-E_{\min}}{\delta}}
  },
\end{equation}
where $\delta>0$ is a smoothing parameter. This leads to the
regularised dose functional
\begin{equation}
  \label{eq:regularised-dose}
  \widetilde D_\delta(x;\theta)
  :=
  \mathbb{E}
  \qb{
    \int_0^T
    S(E_t) \varphi(X_t-x) 
    \mathcal{I}_\delta(E_t,E_{\min})
    \d t
  }.
\end{equation}
In the stochastic straggling regime, the sensitivity analysis below is
carried out for $\widetilde D_\delta(x;\theta)$ rather than for the
stopped functional $D(x;\theta)$ itself.

For the regularised integrand, differentiation may be exchanged with expectation
and the pathwise derivative of $\widetilde D_\delta$ is
\begin{equation}
  \label{eq:dose-sens-full-model}
  \frac{\d}{\d\theta}\widetilde D_\delta(x;\theta)
  =
  \mathbb{E}
  \qb{
    \int_0^T
    \qp{
      \qp{
        \frac{\partial S}{\partial\theta}(E_t)
        +
        \frac{\partial S}{\partial E}(E_t) \partial_\theta E_t
      }
      \mathcal{I}_\delta(E_t,E_{\min})
      +
      S(E_t) 
      \frac{\partial}{\partial\theta}
      \mathcal{I}_\delta(E_t,E_{\min})
    }
    \varphi(X_t-x) \d t
  }.
\end{equation}
Since $\mathcal{I}_\delta$ depends on $\theta$ only through $E_t$, we have
\begin{equation}
  \frac{\partial}{\partial\theta}\mathcal{I}_\delta(E_t,E_{\min})
  =
  \frac{\partial\mathcal{I}_\delta}{\partial E}(E_t,E_{\min}) 
  \partial_\theta E_t
  =
  \frac{\partial_\theta E_t}
       {2\delta \cosh^2 \qp{\tfrac{E_t-E_{\min}}{\delta}}}.
\end{equation}
The smoothing removes the discontinuity at the killing threshold and
yields a dose-type observable whose pathwise sensitivity is well
defined and suitable for numerical approximation.

\section{Pathwise Sensitivities and Observable Gradients}
\label{sec:sensitivity}

Building on \autoref{sec:dose-sensitivity}, we now turn to
derivatives of path-dependent observables with respect to the
parameters $\alpha$, $p$ and $\kappa$ governing the energy dynamics in
\eqref{eq:LogSDE}. Such derivatives are referred to as
\emph{pathwise} or \emph{forward} sensitivities. They are obtained by
differentiating the stochastic system with respect to the parameter,
leading to an augmented system in which the sensitivities evolve
alongside the original state variables and are driven by the same
Brownian path.

In the stochastic straggling regime, the sensitivity process is used
to evaluate gradients of the regularised observable
$\widetilde D(x;\theta)$ introduced in
\autoref{sec:dose-sensitivity}. More generally, the same sensitivity
variables also provide gradients for other regularised energy-based
quantities of interest. For example, one may consider the regularised
survival fraction
\begin{equation}
  Q_{\mathrm{surv}}(T;\theta)
  :=
  \mathbb{E}\qb{\mathcal I(E_T,E_{\min})},
\end{equation}
and the regularised stopping-power profile
\begin{equation}
  Q_{\mathrm{sp}}(t;\theta)
  :=
  \mathbb{E}\qb{S(E_t)\mathcal I(E_t,E_{\min})}.
\end{equation}
Their parameter derivatives are given by
\begin{equation}
  \frac{\d}{\d\theta}Q_{\mathrm{surv}}(T;\theta)
  =
  \mathbb{E}\qb{
    \partial_E\mathcal I(E_T,E_{\min}) \partial_\theta E_T
  },
\end{equation}
and
\begin{equation}
  \frac{\d}{\d\theta}Q_{\mathrm{sp}}(t;\theta)
  =
  \mathbb{E}\qb{
    \qp{
      \partial_\theta S(E_t)
      +
      \partial_E S(E_t) \partial_\theta E_t
    }\mathcal I(E_t,E_{\min})
    +
    S(E_t)\partial_E\mathcal I(E_t,E_{\min}) \partial_\theta E_t
  }.
\end{equation}
Thus, once the pathwise sensitivity $\partial_\theta E_t$ is
available, gradients of several observables follow directly by the
chain rule.

\subsection{Sensitivity Equations}

For pathwise parameter sensitivities we work with the model
\eqref{eq:SDE2} with constant angular diffusion,
\begin{equation}
  \label{eq:model-for-sensitivities}
  \begin{split}
    \d X_t &= \Omega_t \d t,
    \\
    \d E_t &= -S(E_t)\d t + \sqrt{T(E_t)} \d W_t^E,
    \\
    \d \Omega_t &= \sqrt{2\epsilon_0}\circ \d W_t^S,
  \end{split}
\end{equation}
where $W^E$ and $W^S$ are independent Brownian motions. We are
interested in the energy sensitivities with respect to
$\theta\in\{\alpha,p,\kappa\}$, defined by
\begin{equation}
  J_t^\theta
  :=
  \partial_\theta E_t,
  \qquad
  J_0^\theta=\partial_\theta E_0
  \quad
  (=0 \text{ if } E_0 \text{ is fixed}).
\end{equation}
Differentiating the energy SDE with respect to $\theta$ yields the
linear SDE
\begin{equation}
  \label{eq:energy-sensitivity-SDE}
  \d J_t^\theta
  =
  -\qp{
    \partial_\theta S(E_t)
    +
    \partial_E S(E_t) J_t^\theta
  }\d t
  +
  \qp{
    \partial_\theta \sqrt{T(E_t)}
    +
    \partial_E \sqrt{T(E_t)} J_t^\theta
  }\d W_t^E.
\end{equation}
For each $\theta$, this equation is driven by the \emph{same}
Brownian motion $W^E$ as the primary energy process. This shared-noise
coupling is what allows the sensitivity process to be combined with
the regularised observable representations of
\autoref{sec:dose-sensitivity}.

\subsection{Explicit Derivatives of the Coefficients}

Since only the energy dynamics depend on the parameters
$\alpha$, $p$ and $\kappa$, it suffices to compute derivatives of the
drift $S(E)$ and diffusion $\sqrt{T(E)}$ with respect to $E$ and to
the model parameters. These determine both the coefficients of the
sensitivity equations and the inhomogeneous forcing terms.

Derivatives with respect to the energy variable are
\begin{equation}
  \begin{split}
    \partial_E S(E)
    &=
    \frac{1-p}{\alpha p} E^{-p}
    =
    \frac{1-p}{E}S(E),
    \\
    \partial_E \sqrt{T(E)}
    &=
    \qp{1-\tfrac{p}{2}}
    \sqrt{\frac{\kappa}{\alpha p}} E^{-p/2}
    =
    \frac{1}{2}\frac{T'(E)}{\sqrt{T(E)}}.
  \end{split}
\end{equation}
With respect to $\alpha$,
\begin{equation}
  \begin{split}
    \partial_\alpha S(E)
    &=
    -\frac{1}{\alpha}S(E),
    \\
    \partial_\alpha \sqrt{T(E)}
    &=
    -\frac{1}{2\alpha}\sqrt{T(E)}.
  \end{split}
\end{equation}
With respect to $p$,
\begin{equation}
  \begin{split}
    \partial_p S(E)
    &=
    -\frac{1}{p}S(E)
    -
    (\log E)S(E)
    \\
    &=
    -\frac{E^{1-p}}{\alpha p^2}
    -
    \frac{\log E}{\alpha p}E^{1-p},
    \\
    \partial_p \sqrt{T(E)}
    &=
    -\frac{1}{2p}\sqrt{T(E)}
    -
    \frac{1}{2}(\log E)\sqrt{T(E)}.
  \end{split}
\end{equation}
With respect to $\kappa$,
\begin{equation}
  \begin{split}
    \partial_\kappa S(E)
    &= 0,
    \\
    \partial_\kappa \sqrt{T(E)}
    &= \frac{1}{2\kappa}\sqrt{T(E)}.
  \end{split}
\end{equation}

\begin{remark}[Coupling with the state variables]
  Although the sensitivity variable $J_t^\theta$ carries no geometric
  constraint of its own, its discretisation must remain coupled to the
  state approximation. In particular, the same Brownian increments
  used in the energy update must also be used in the sensitivity
  update. This preserves the pathwise coupling between the state and
  sensitivity variables and is the basis for the observable-gradient
  formulas above.
\end{remark}

\subsection{Discretisation of the Sensitivity Equations}

To simulate sensitivities for $\theta\in\{\alpha,p,\kappa\}$, we add
one scalar sensitivity equation to the system
\eqref{eq:model-for-sensitivities}. Each such variable is coupled to
the state only through the energy equation and evolves independently
of the others once the state path is fixed.

\subsubsection{Euler--Maruyama discretisation}

Given $Y_n$ and $J_n^\theta$ at time $t_n$, one geometric Euler step
with $\xi_n^E\sim\mathcal N(0,h)$ first advances the state
$(X_n,Y_n,\Omega_n)$ through \eqref{eq:LogEulerExp}, and then updates
the sensitivity by
\begin{equation}
  \label{eq:euler-logY-and-J}
  \begin{split}
    E_n &= \exp(Y_n),
    \\
    J_{n+1}^\theta
    &=
    J_n^\theta
    -
    \qp{
      \partial_\theta S(E_n)
      +
      \partial_E S(E_n)J_n^\theta
    }h
    +
    \qp{
      \partial_\theta \sqrt{T(E_n)}
      +
      \partial_E \sqrt{T(E_n)}J_n^\theta
    }\xi_n^E,
  \end{split}
\end{equation}
where the same increment $\xi_n^E$ is reused from the state update.

\subsubsection{Milstein discretisation}

For higher accuracy, we instead advance $(X_n,Y_n,\Omega_n)$ through
\eqref{eq:LogMilsteinExp}, set
\begin{equation}
  E_n=\exp(Y_n),
\end{equation}
and then update the energy sensitivity using the same Brownian
increment $\xi_n^E\sim\mathcal N(0,h)$ as in the log-energy Milstein
step:
\begin{equation}
  \label{eq:sensitivity-milstein}
  \begin{split}
    J_{n+1}^\theta
    &=
    J_n^\theta
    -
    \qp{
      \partial_\theta S(E_n)
      +
      \partial_E S(E_n)J_n^\theta
    }h
    +
    \qp{
      \partial_\theta \sqrt{T(E_n)}
      +
      \partial_E \sqrt{T(E_n)}J_n^\theta
    }\xi_n^E
    \\
    &\quad
    +
    \frac12
    \partial_E\qp{
      \sqrt{T(E_n)}
      \qp{
        \partial_\theta \sqrt{T(E_n)}
        +
        \partial_E \sqrt{T(E_n)}J_n^\theta
      }
    }
    \qp{(\xi_n^E)^2-h}.
  \end{split}
\end{equation}
This is the standard Milstein discretisation for a multidimensional SDE driven by one-dimensional Brownian
motion (see e.g. Section~10.3 of \cite{Kloeden1992}), applied here to the coupled
state-sensitivity system with common noise.

\begin{proposition}[Consistent augmented state-sensitivity discretisation]
  \label{thm:strong-augmented}
  Let the state variables $(X_n,Y_n,\Omega_n)$ be advanced either by
  the geometric Euler scheme \eqref{eq:LogEulerExp} or by the
  higher-order scheme \eqref{eq:LogMilsteinExp}. For a fixed parameter
  $\theta\in\{\alpha,p,\kappa\}$, let $J_n^\theta$ be updated by
  \eqref{eq:euler-logY-and-J} in the Euler case and by
  \eqref{eq:sensitivity-milstein} in the Milstein case, always reusing
  the same Brownian increment $\xi_n^E$ as in the corresponding
  energy update.

  Then the sensitivity updates are obtained by differentiating the
  corresponding energy discretisations with respect to $\theta$ along
  the same realised Brownian path. In particular, the augmented
  numerical scheme is consistent with the formal sensitivity equation
  \eqref{eq:energy-sensitivity-SDE} and preserves the required
  pathwise coupling between state and sensitivity variables.
\end{proposition}

\begin{proof}
  In the Euler case, differentiate the discrete energy update
  \eqref{eq:LogEulerExp} with respect to $\theta$ while holding the
  realised increment $\xi_n^E$ fixed. Applying the chain rule to the
  dependence of $S(E_n)$ and $T(E_n)$ on both $E_n$ and $\theta$
  yields exactly \eqref{eq:euler-logY-and-J}. The same argument
  applies to the Milstein case: differentiating the log-energy
  Milstein update \eqref{eq:LogMilsteinExp} term by term with respect
  to $\theta$, again with the same realised increment $\xi_n^E$,
  gives \eqref{eq:sensitivity-milstein}. Since these are precisely the
  discretisations obtained from the formal sensitivity equation using
  the same Brownian path, the resulting augmented scheme is
  pathwise-coupled and consistent.
\end{proof}

\begin{remark}[Use in observable gradients]
  The result above concerns the pathwise sensitivity process itself.
  Gradients of regularised observables are then obtained by combining
  $J_n^\theta$ with the chain-rule representations from
  \autoref{sec:dose-sensitivity}. In particular, the same simulated
  sensitivity path may be reused to estimate derivatives of
  $\widetilde D(x;\theta)$, $Q_{\mathrm{surv}}(T;\theta)$ and
  $Q_{\mathrm{sp}}(t;\theta)$ within a common numerical framework.

  We do not pursue a separate convergence proof for the observable
  gradients here, although these results do follow from arguments in
  \cite{giles2024strong}. Rather, the purpose of this section is to
  show that the discretised sensitivity equations are naturally
  induced by the structure-preserving discretisations of the
  underlying state dynamics and remain synchronised with them along
  each sampled path.
\end{remark}

\section{Numerical experiments}
\label{sec:numerics}

\subsection{Dose computations}

Here we present numerical dose calculations comparing three variants
of the model. These experiments illustrate how model parameters shape
the dose distribution, motivating the quantitative sensitivity
analysis in \autoref{sec:sensitivity}. They also highlight the effect
of using the extended model in which the angular diffusion depends on
energy in accordance with Moli\`ere theory.

We first examine how introducing energy straggling alters the dose. We
then study angular diffusion, initially with a constant coefficient
$\epsilon_0$ and subsequently with an energy dependent coefficient
$\epsilon(E)$. As a baseline for each comparison, we include the dose
obtained with
$(\alpha,p,\kappa,\epsilon_0)=(0.0022, 1.77, 0, 0.005)$.

Unless stated otherwise, simulations use an initial beam energy of
$62$~MeV, an initial position
$(x_0,y_0)=(0.0~\text{cm}, 2.0~\text{cm})$, a Gaussian energy spread
of $1\%$ and a Gaussian transverse half width of $0.1$~cm. The initial
direction is parallel to the $x$-axis (longitudinal). The spatial grid
uses $N_x=200$ points along $x$ (along the beam) and $N_y=50$ points
along $y$ (transverse). The SDE system is advanced with time step
$h=0.005$ and $N=200,000$ independent paths are averaged to estimate
the expectation in \autoref{eq:dose}.

\begin{example}[Effect of energy straggling]\label{ex:straggling}
  We examine how the energy-straggling amplitude $\kappa$ influences
  the dose. All other parameters are fixed at $\alpha=0.0022$,
  $p=1.77$ and $\epsilon_0=0.005$, so that only $\kappa$ varies.

  Since the straggling term in \autoref{eq:SDE2} controls fluctuations
  in energy loss, larger $\kappa$ should broaden the Bragg peak and
  lower its maximum. This is observed in \autoref{fig:straggling}:
  moving from panels~(a) to~(c), the peak widens and its height
  decreases while the overall depth range remains similar.

  \begin{figure}[h!]
    \centering
    \captionsetup[subfigure]{justification=centering}
    \begin{subfigure}[b]{0.32\textwidth}
      \includegraphics[width=\textwidth]{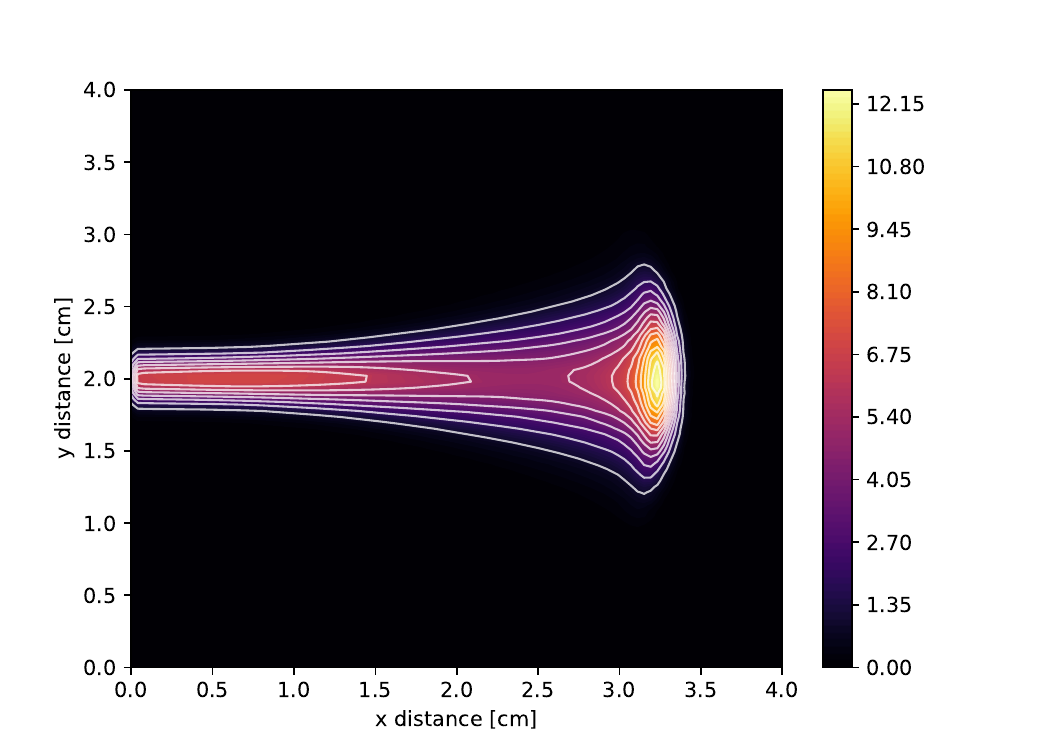}
      \caption{No straggling ($\kappa=0$); $\epsilon_0$ constant.}
      \label{fig:straggling1}
    \end{subfigure}
  \hfill
  \begin{subfigure}[b]{0.32\textwidth}
    \includegraphics[width=\textwidth]{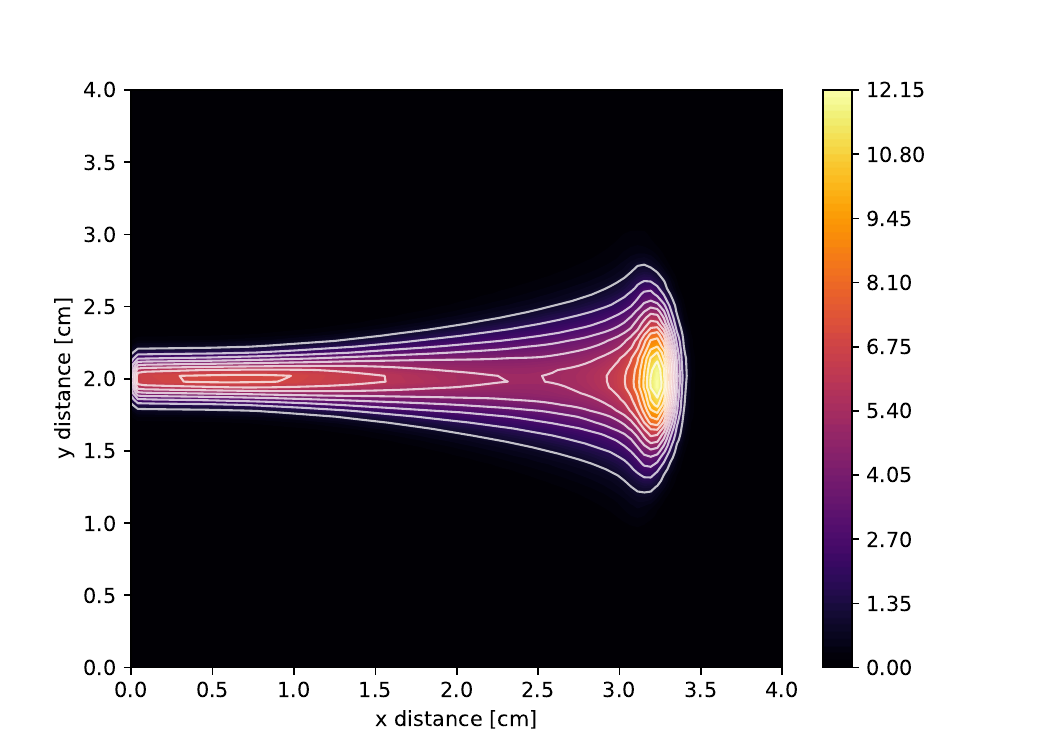}
    \caption{Moderate straggling ($\kappa=4\times10^{-5}$); $\epsilon_0$ constant.}
    \label{fig:straggling2}
  \end{subfigure}
  \hfill
  \begin{subfigure}[b]{0.32\textwidth}
    \includegraphics[width=\textwidth]{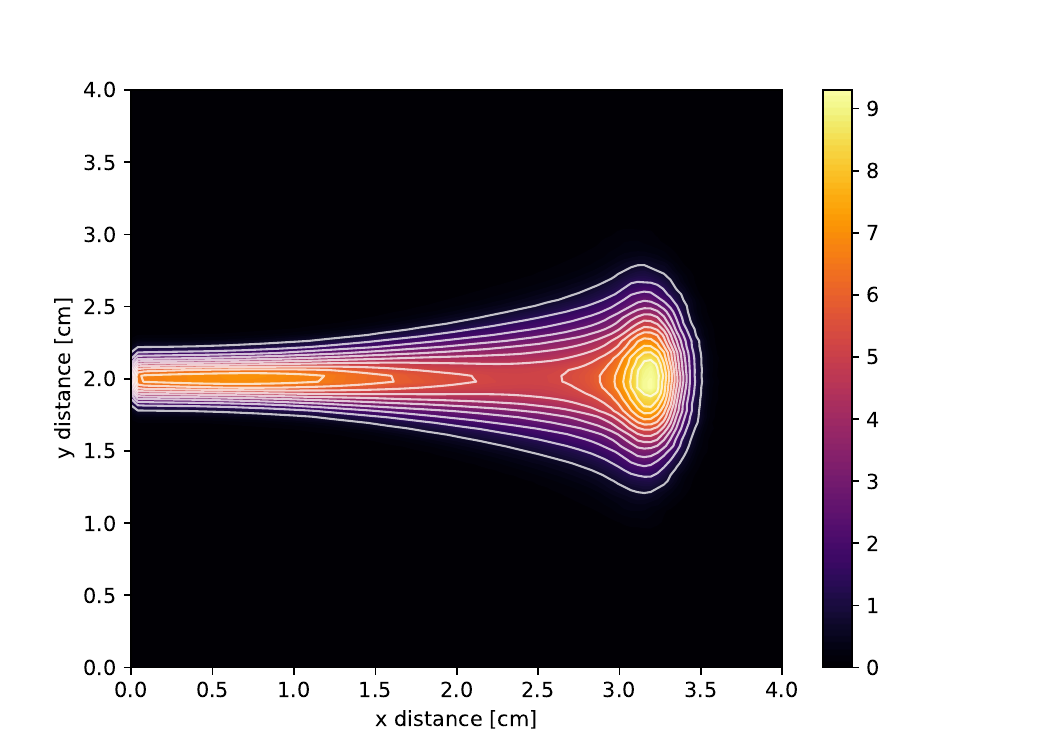}
    \caption{Strong straggling ($\kappa=10^{-3}$); $\epsilon_0$ constant.}
    \label{fig:straggling3}
  \end{subfigure}
  \caption{Example~\ref{ex:straggling}. Impact of the
    energy-straggling amplitude $\kappa$ on the dose map. Increasing
    $\kappa$ broadens the Bragg peak and reduces its maximum,
    consistent with greater variance in energy deposition. In all
    panels the angular diffusion is held fixed at
    $\epsilon_0=0.005$.}
  \label{fig:straggling}
\end{figure}
\end{example}

\begin{example}[Effect of angular diffusion]\label{ex:eps-const}
  We examine the role of the angular diffusion amplitude $\epsilon_0$
  in \autoref{eq:SDE1}. Since angular diffusion governs the lateral
  spread of particle trajectories, we expect a narrower, more
  collimated dose profile when $\epsilon_0$ is small, and a broader,
  fanned-out profile when $\epsilon_0$ is large. This behaviour is
  confirmed in \autoref{fig:constant-eps}. As $\epsilon_0$ increases,
  the distribution remains narrow near the entry point but spreads
  significantly near the distal edge and Bragg peak. We also note the
  dip in the dose just before the peak, a feature that has been
  observed experimentally in proton beams \cite{reaz2022sharp}.

  \begin{figure}[h!]
    \centering
    \captionsetup[subfigure]{justification=centering}
    \begin{subfigure}[b]{0.32\textwidth}
      \includegraphics[width=\textwidth]{plots/dose_plots/dose_no_straggling.pdf}
      \caption{$\epsilon_0=0.005$ (reference). \phantom{blurby blurb}}
      \label{fig:eps1}
    \end{subfigure}
    \hfill
    \begin{subfigure}[b]{0.32\textwidth}
      \includegraphics[width=\textwidth]{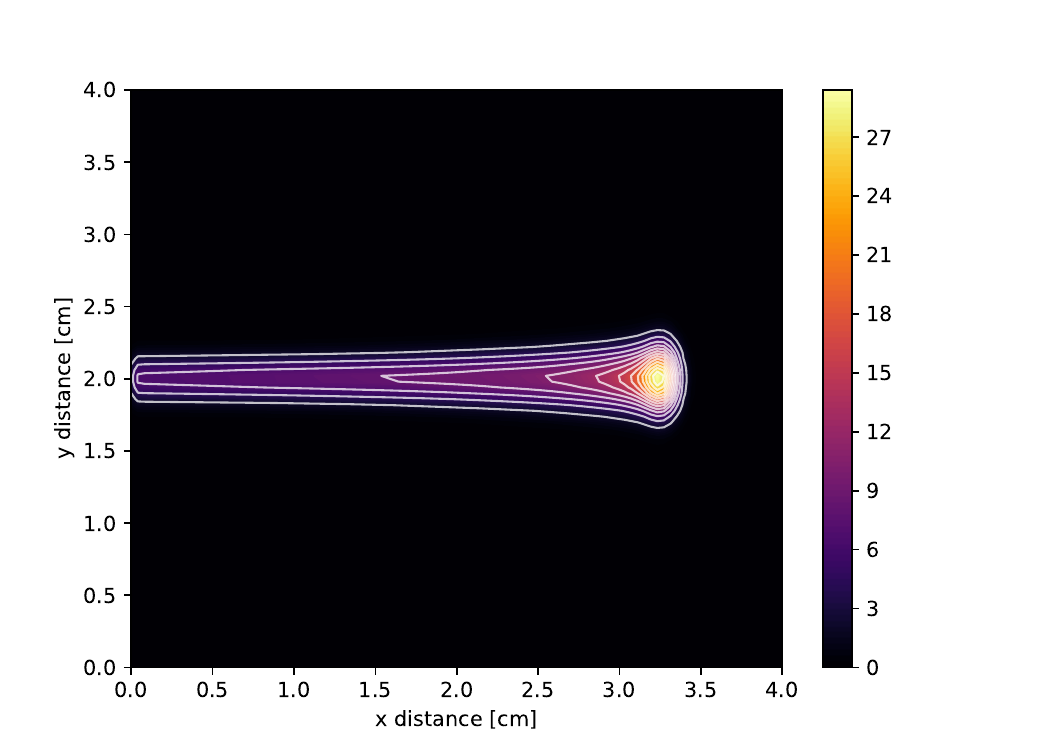}
      \caption{$\epsilon_0=0.0005$ (small angular diffusion).}
      \label{fig:eps2}
    \end{subfigure}
    \hfill
    \begin{subfigure}[b]{0.32\textwidth}
      \includegraphics[width=\textwidth]{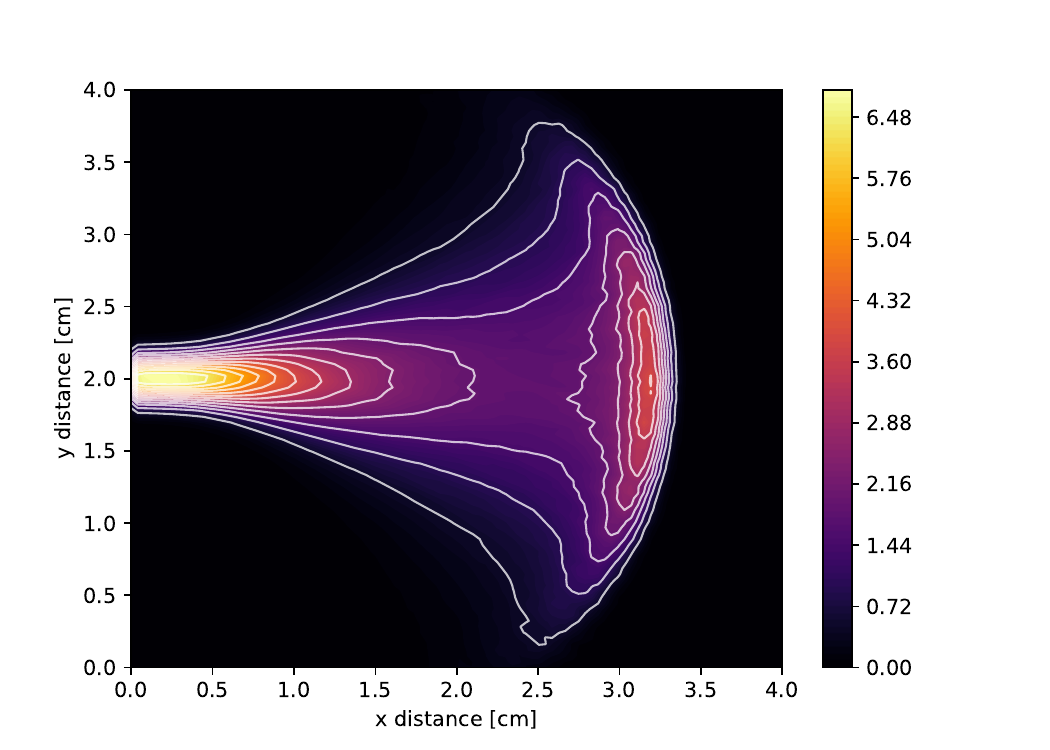}
      \caption{$\epsilon_0=0.05$ (large angular diffusion).}
      \label{fig:eps3}
    \end{subfigure}
    \caption{Example~\ref{ex:eps-const}. Impact of constant angular
      diffusion on the dose map. A small $\epsilon_0$ yields a
      straight, narrow profile, while a large $\epsilon_0$ produces a
      broad fan-like spread with a characteristic dip before the Bragg
      peak.}
  \label{fig:constant-eps}
\end{figure}
\end{example}

\begin{example}[Effect of energy-dependent angular diffusion]
  \label{ex:moliere}
  We now compare dose distributions obtained with constant angular
  diffusion and with the energy-dependent diffusion coefficient
  $\epsilon(E)$ given by \autoref{eq:moliere-coef}, setting
  $\kappa=0$. To make the comparison, the parameters $\bar{\epsilon}$
  and $\epsilon_c$ are chosen so that the diffusion coefficient at the
  initial energy $E_0=62 \mathrm{MeV}$ matches that of the constant
  model. Specifically, fixing $\epsilon_c=5.0$, we choose
  $\bar{\epsilon}$ such that
  \begin{equation}
    \frac{\bar{\epsilon}}{E_0^2 + \epsilon_c^2} = \epsilon_0.
  \end{equation}
  
  The parameter values in \autoref{fig:moliere2} and
  \autoref{fig:moliere3} are chosen so that $\epsilon(E_0)$ equals
  $0.005$ and $0.025$ respectively. As shown in \autoref{fig:moliere},
  the resulting dose maps are broadly similar to those from the
  constant-diffusion model. However, when the diffusion coefficient is
  larger (Figure~\ref{fig:moliere3}), the distribution near the distal
  edge is more spread out than in the constant case
  (cf.~Figure~\ref{fig:eps3}), even without energy straggling.

  \begin{figure}[h!]
    \centering
    \captionsetup[subfigure]{justification=centering}
    \begin{subfigure}[b]{0.32\textwidth}
      \includegraphics[width=\textwidth]{plots/dose_plots/dose_no_straggling.pdf}
      \caption{Constant angular diffusion, $\epsilon_0=0.005$.}
      \label{fig:moliere1}
    \end{subfigure}
    \hfill
    \begin{subfigure}[b]{0.32\textwidth}
      \includegraphics[width=\textwidth]{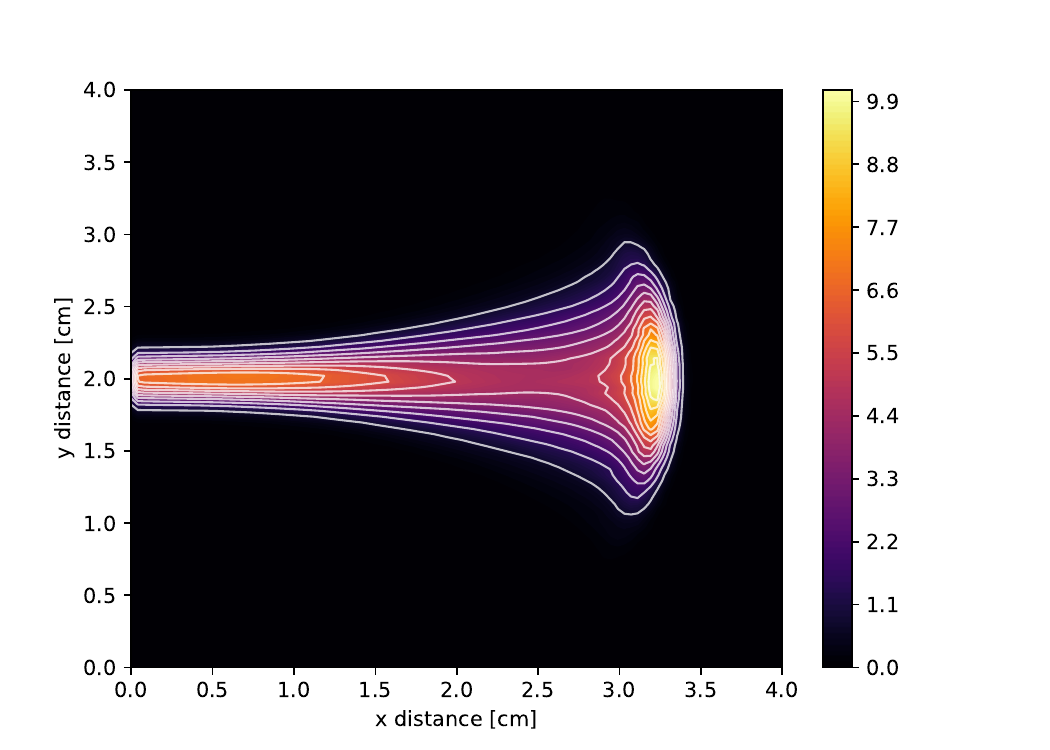}
      \caption{Energy-dependent diffusion.\newline $\bar{\epsilon}=19.3$, $\epsilon_c=5.0$.}
      \label{fig:moliere2}  
    \end{subfigure}
    \hfill
    \begin{subfigure}[b]{0.32\textwidth}
      \includegraphics[width=\textwidth]{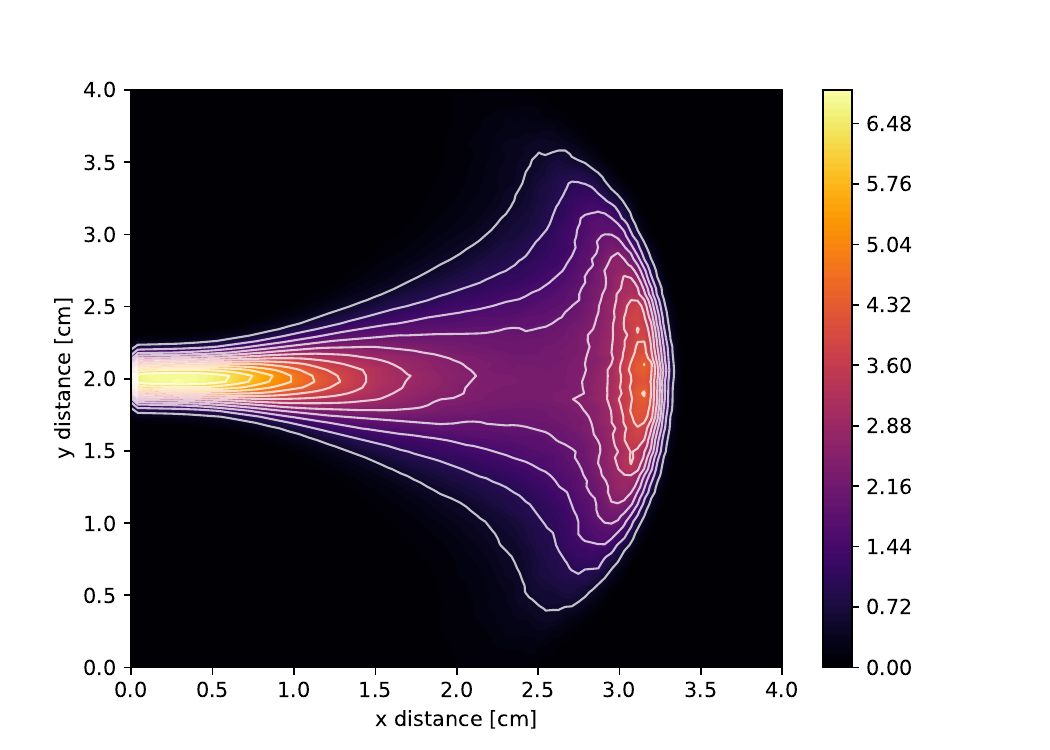}
      \caption{Energy-dependent diffusion.\newline $\bar{\epsilon}=96.7$, $\epsilon_c=5.0$.}
      \label{fig:moliere3}
    \end{subfigure}
    \caption{Example~\ref{ex:moliere}. Comparison of constant and
      energy-dependent angular diffusion. Larger diffusion
      coefficients produce more fanned-out dose distributions,
      particularly near the distal edge.}
  \label{fig:moliere}
\end{figure}
\end{example}

\subsection{Sensitivity computations}

We now present numerical experiments illustrating the computation of
pathwise sensitivities using the structure-preserving discretisations
developed above. Our focus is on dose-related observables. In the
no-straggling case we use the exact stopped-dose representation from
Section~\ref{sec:dose-sensitivity}, while in the stochastic-straggling
regime we use the regularised surrogate observable introduced
there. We compare the resulting pathwise estimators with
finite-difference benchmarks and examine their empirical behaviour.

As for the dose computations above, unless stated otherwise, simulations use an initial beam energy of
$62$~MeV, an initial position
$(x_0,y_0)=(0.0~\text{cm}, 2.0~\text{cm})$, a Gaussian energy spread
of $1\%$ and a Gaussian transverse half width of $0.1$~cm. The initial
direction is parallel to the $x$-axis (longitudinal). Here, the spatial grid
uses $N_x=100$ points along $x$ (along the beam) and $N_y=50$ points
along $y$ (transverse). The SDE system is advanced with time step
$h=0.01$ and $N=200,000$ independent paths are averaged.

\begin{remark}[Interpretation of sensitivities]
  The sensitivity profiles quantify how the expected dose distribution
  shifts in response to perturbations of the stopping power
  parameters. A negative sensitivity means that increasing the
  parameter decreases dose at that spatial location, while a positive
  sensitivity indicates dose amplification. For example, the negative
  region on the proximal side and the positive region on the distal
  side of the Bragg peak show that increasing $\alpha$ or $p$ shortens
  the particle range and moves the peak towards the entrance, while
  decreasing them lengthens the range.

  These profiles provide local information on where uncertainties in
  physical parameters most strongly affect dose deposition. Regions
  with large sensitivity magnitude are those where parameter
  uncertainty is most critical for treatment planning and robustness
  analysis, while regions with near-zero sensitivity are less
  affected. Thus, the sensitivity plots can be read as maps of dose
  susceptibility to modelling error in the stopping power law.
\end{remark}

To assess the empirical accuracy of the pathwise method, we compare it
with central finite-difference estimates. For a
parameter-dependent observable $\Phi_\theta$, the finite-difference
gradient estimator is
\begin{equation}\label{eq:fd-estimator}
  \frac{\partial}{\partial\theta}\mathbb{E}\qb{\Phi_\theta}
  \approx
  \frac{
    \mathbb{E}\qb{\Phi_{\theta+\Delta\theta}}
    -
    \mathbb{E}\qb{\Phi_{\theta-\Delta\theta}}
  }{2\Delta\theta}.
\end{equation}
In the no-straggling case, $\Phi_\theta$ is the stopped-dose
functional from Section~\ref{sec:dose-sensitivity}; in the full
model, $\Phi_\theta$ is the corresponding regularised surrogate
observable.

In contrast to the pathwise estimator, the finite-difference estimator
exhibits noticeable bias for nonzero $\Delta\theta$, and the usual 
bias-variance tradeoff as $\Delta\theta \to 0$.

\begin{example}[Parameter sensitivities under the no-straggling model]
  \label{ex:no-straggling-model-sensitivities}
  We compute parameter sensitivities of the dose distribution under
  the model with $\kappa=0$, as specified in \autoref{eq:SDE1}. In
  this case, the equation governing energy evolution reduces from an
  SDE to an ODE, which admits a closed-form solution. The dose and
  pathwise sensitivity computations therefore reduce to numerical
  evaluation of \autoref{eq:no-straggling-dose} and
  \autoref{eq:dose-sens-no-straggling}. Finite difference
  sensitivities are obtained according to \autoref{eq:fd-estimator},
  evaluating the dose functional \autoref{eq:no-straggling-dose} at
  two perturbed values of the parameter $\theta \in \{\alpha,p\}$. For
  the results shown, we fix $\alpha=0.0022$, $p=1.77$,
  $\epsilon_0=0.005$, and choose $\Delta\alpha = 0.1\alpha$, $\Delta p
  = 0.01p$ for the finite difference estimates.
  
  Figures \ref{fig:dose-sens-alpha-no-straggling} and
  \ref{fig:dose-sens-p-no-straggling} display the resulting
  sensitivity profiles for $\alpha$ and $p$ respectively. In both
  cases, the finite difference estimators underestimate the magnitude
  of the sensitivity compared to the pathwise estimators. They also
  produce sensitivity profiles that are more spatially spread
  out. Both effects stem from the bias of the finite difference
  estimator and the numerical diffusion it introduces, which smooths
  and attenuates sharp features.

  Qualitatively, both estimators capture the same behaviour. Negative
  gradients on the proximal side of the dose peak and positive
  gradients on the distal side. Since $\alpha$ and $p$ appear in the
  stopping power $S(E)$, which governs the mean rate of energy loss
  and hence the particle range, this behaviour is consistent, reducing
  the stopping power increases the range, shifting the peak
  deeper. Finally, the magnitude of the sensitivity with respect to
  $\alpha$ is orders of magnitude larger than that for $p$.

\begin{figure}[h!]
  \captionsetup[subfigure]{justification=centering}
  \centering
  \begin{subfigure}[b]{0.32\textwidth}
    \includegraphics[width=\textwidth]{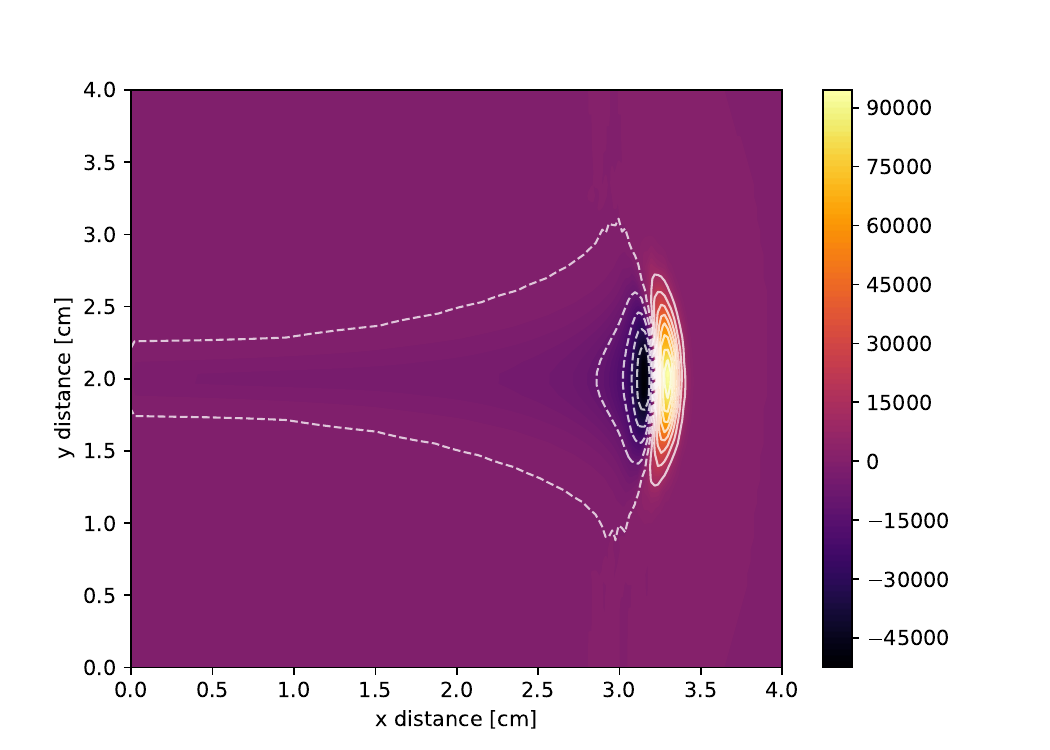}
    \caption{Pathwise sensitivity.}\label{fig:alpha-pathwise}
  \end{subfigure}
  \begin{subfigure}[b]{0.32\textwidth}
    \includegraphics[width=\textwidth]{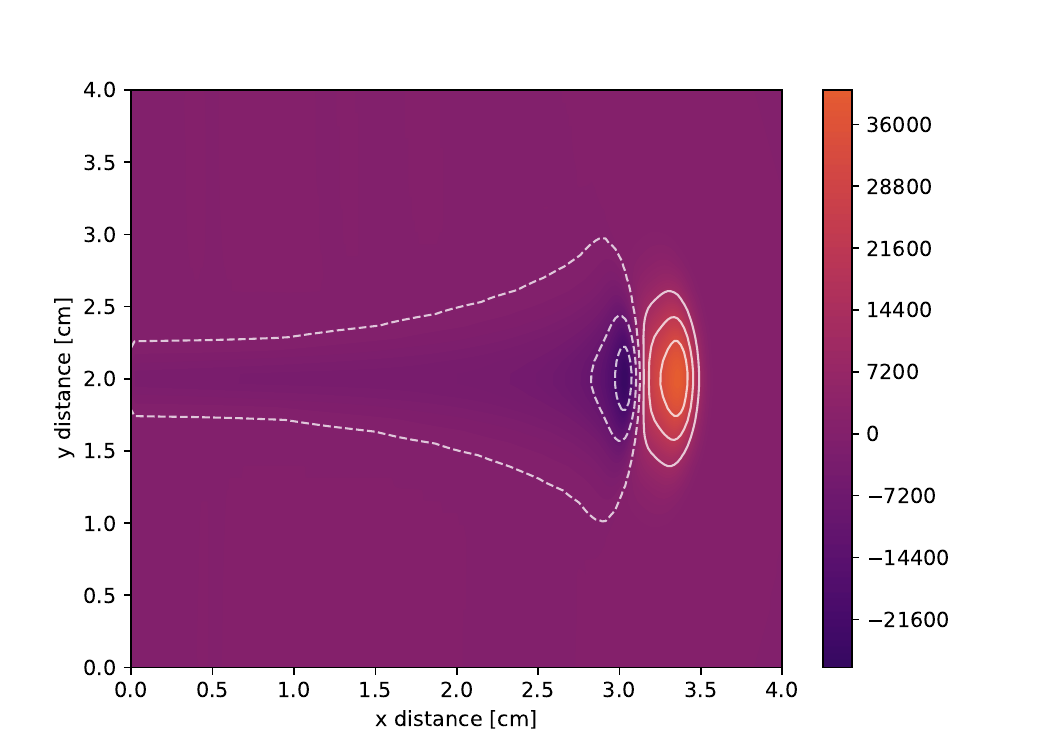}
    \caption{Finite difference sensitivity.}\label{fig:alpha-fd}
  \end{subfigure}
  \begin{subfigure}[b]{0.32\textwidth}
    \includegraphics[width=\textwidth]{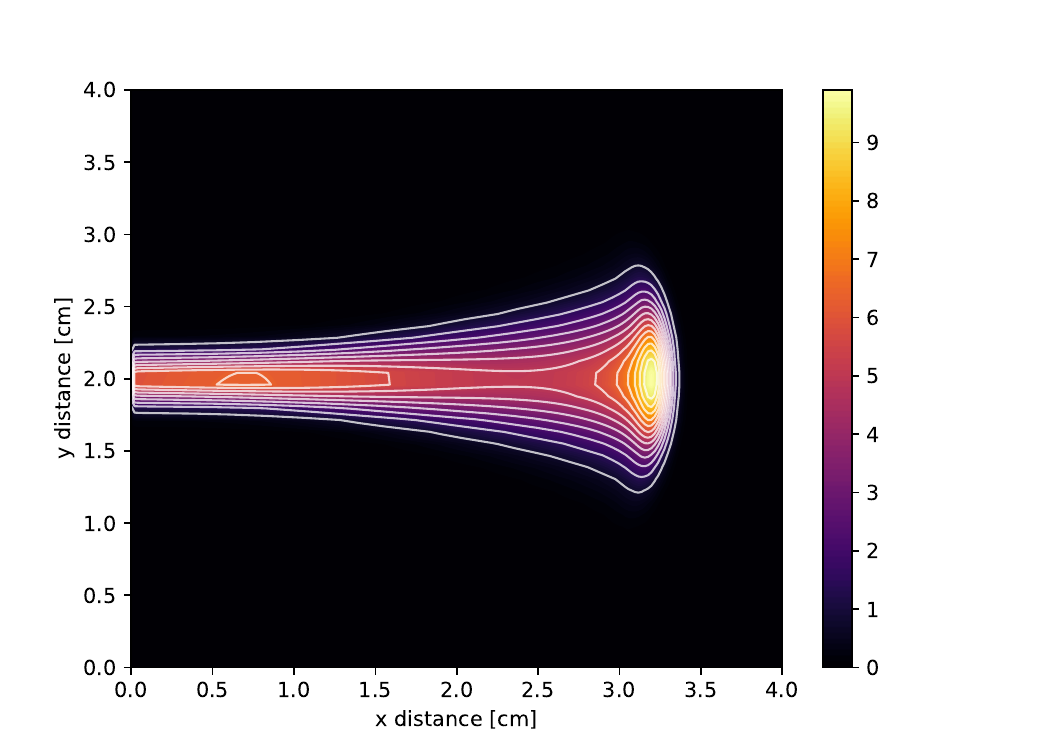}
    \caption{Reference dose.}
  \end{subfigure}
  \caption{Example \ref{ex:no-straggling-model-sensitivities}.
    Sensitivity with respect to $\alpha$. The finite difference
    estimator underestimates the gradient magnitude and produces a
    more diffuse profile compared to the pathwise
    method.}\label{fig:dose-sens-alpha-no-straggling}
\end{figure}

\begin{figure}[h!]
  \captionsetup[subfigure]{justification=centering}
  \centering
  \begin{subfigure}[b]{0.32\textwidth}
    \includegraphics[width=\textwidth]{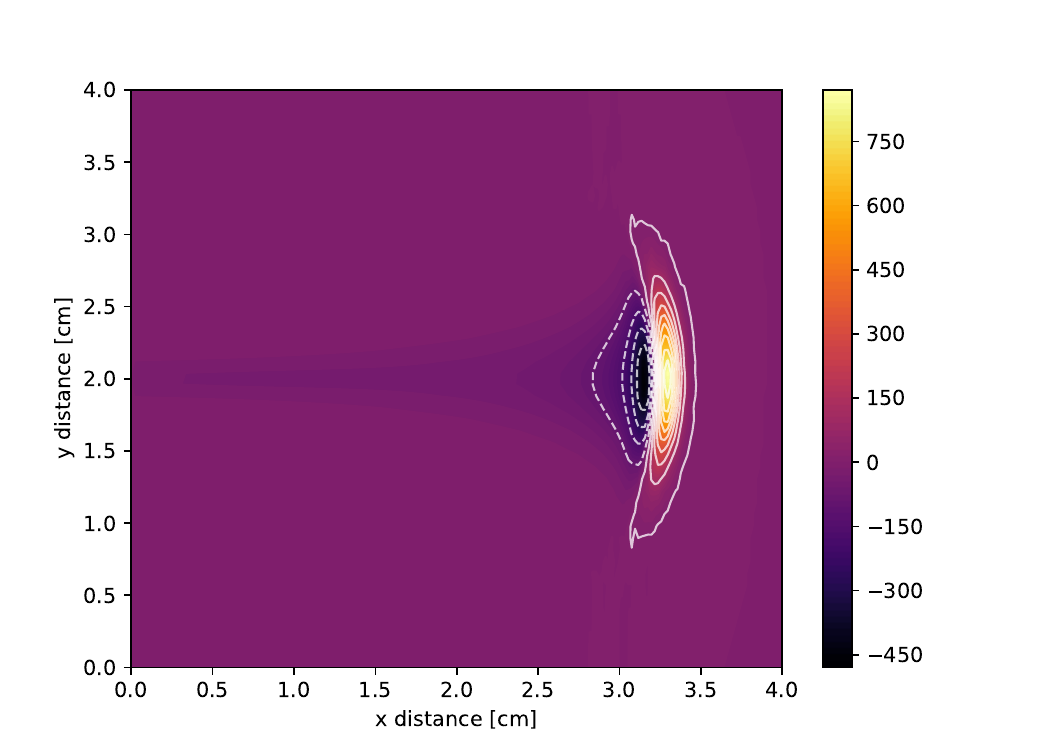}
    \caption{Pathwise sensitivity.}\label{fig:p-pathwise}
  \end{subfigure}
  \begin{subfigure}[b]{0.32\textwidth}
    \includegraphics[width=\textwidth]{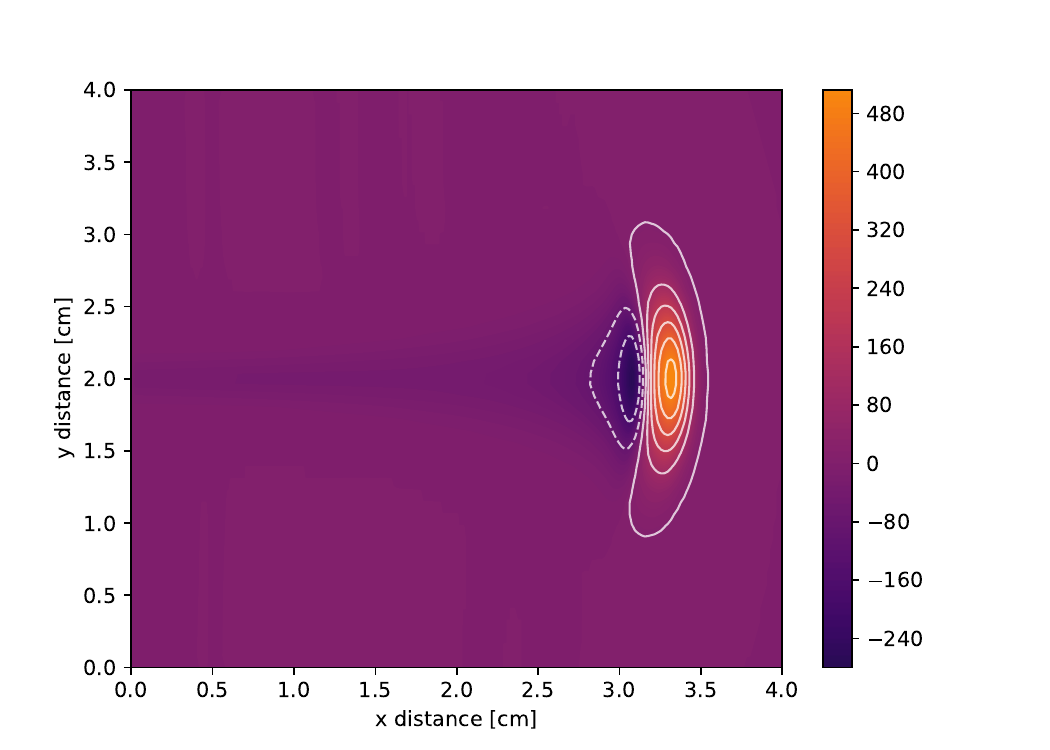}
    \caption{Finite difference sensitivity.}\label{fig:p-fd}
  \end{subfigure}
  \begin{subfigure}[b]{0.32\textwidth}
    \includegraphics[width=\textwidth]{plots/dose_sens/no_straggling/dose_ref.pdf}
    \caption{Reference dose.}
  \end{subfigure}
  \caption{Example \ref{ex:no-straggling-model-sensitivities}.
    Sensitivity with respect to $p$. The same qualitative features are
    captured by both methods, but the finite difference estimator
    underestimates the gradient magnitude and introduces artificial
    spreading.}\label{fig:dose-sens-p-no-straggling}
\end{figure}
\end{example}

\begin{example}[Parameter sensitivities under the full model]
  \label{ex:full-model-sensitivities}
  In this regime, both the pathwise and finite-difference computations
  are carried out for the regularised surrogate observable $\widetilde
  D_\delta(x;\theta)$ from Section~\ref{sec:straggling_sens}, using
  the same smoothing parameters.
  
  We now compute sensitivities using the full model with nonzero
  $\kappa$. In this case, the pathwise estimator for dose sensitivity
  is given by \autoref{eq:dose-sens-full-model}. The parameters of the
  mollified indicator function in \autoref{eq:mollified-indicator} are
  chosen as $\delta=0.5$ and $E_{\min}=2.0$. Model parameters are
  fixed as $\alpha=0.0022$, $p=1.77$, $\epsilon_0=0.005$,
  $\kappa=0.001$, while finite difference estimates use $\Delta \alpha
  = 0.1\alpha$, $\Delta p = 0.01p$, and $\Delta \kappa = 0.1\kappa$.

  Figures \ref{fig:dose-sens-alpha}, \ref{fig:dose-sens-p}, and
  \ref{fig:dose-sens-kappa} show the dose profile and sensitivities
  with respect to $\alpha$, $p$, and $\kappa$. In all cases, the
  finite difference estimates underestimate the gradient magnitude
  compared to the pathwise estimates, most clearly for $\alpha$, but
  also for $p$ and $\kappa$. In these experiments, the
  finite-difference estimates are smoother and of smaller magnitude
  than the pathwise estimates, consistent with finite-difference bias
  and Monte Carlo noise.
  
  Qualitatively, both estimators capture the same structure. For
  $\alpha$ and $p$, the gradient is negative before the dose peak and
  positive near or beyond the peak, reflecting their role in the
  stopping power $S(E)$ and hence the particle range. For $\kappa$,
  the gradient is negative at the peak and positive on either
  side. Since larger $\kappa$ widens the dose peak while reducing its
  height (as seen in Example \ref{ex:straggling}), this shape is
  consistent with the underlying physics.

  \begin{figure}[h!]
    \captionsetup[subfigure]{justification=centering}
    \centering
    \begin{subfigure}[b]{0.32\textwidth}
      \includegraphics[width=\textwidth]{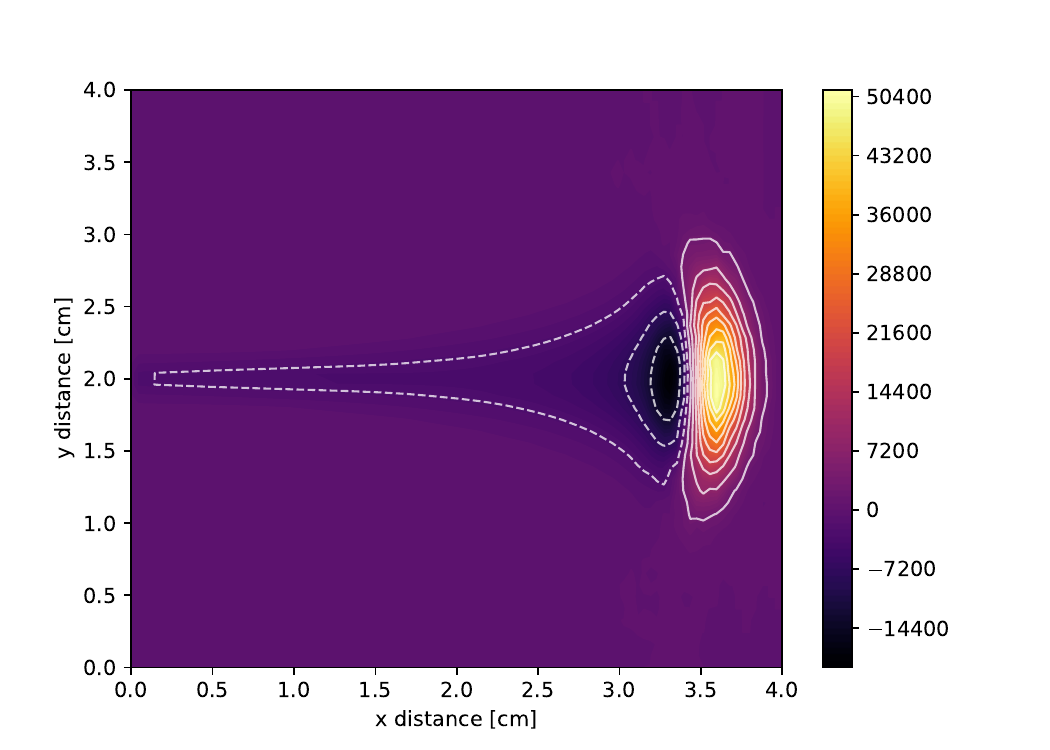}
      \caption{Pathwise sensitivity.}\label{fig:alpha-pathwise-full}
    \end{subfigure}
    \begin{subfigure}[b]{0.32\textwidth}
      \includegraphics[width=\textwidth]{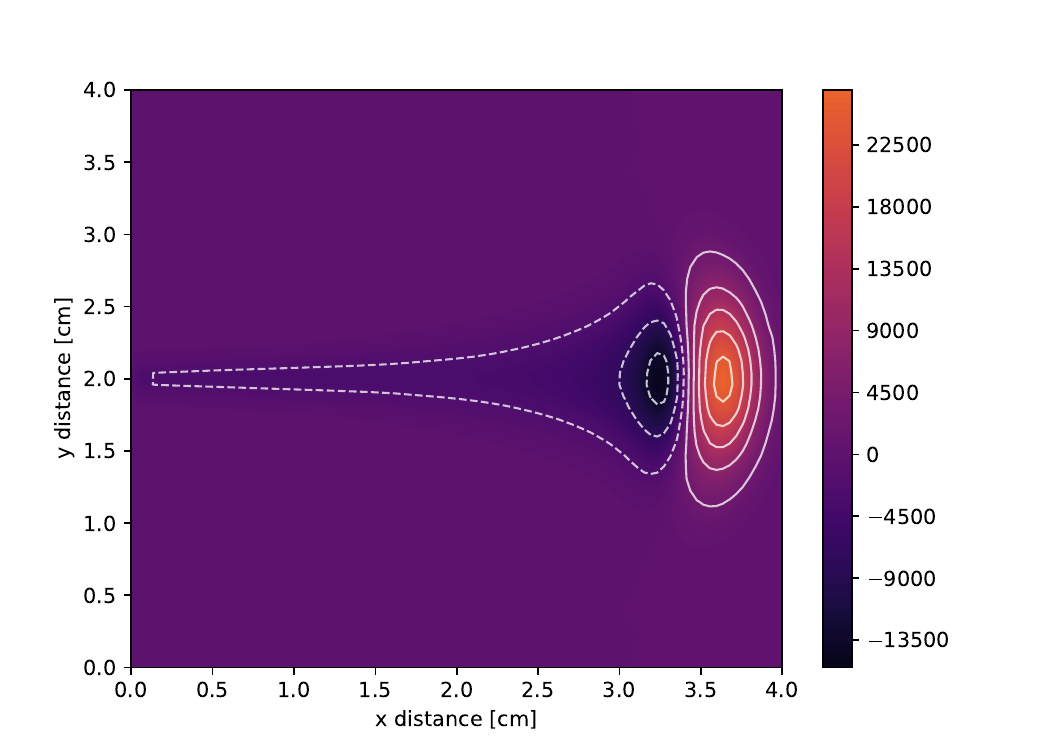}
      \caption{Finite difference sensitivity.}\label{fig:alpha-fd-full}
    \end{subfigure}
    \begin{subfigure}[b]{0.32\textwidth}
      \includegraphics[width=\textwidth]{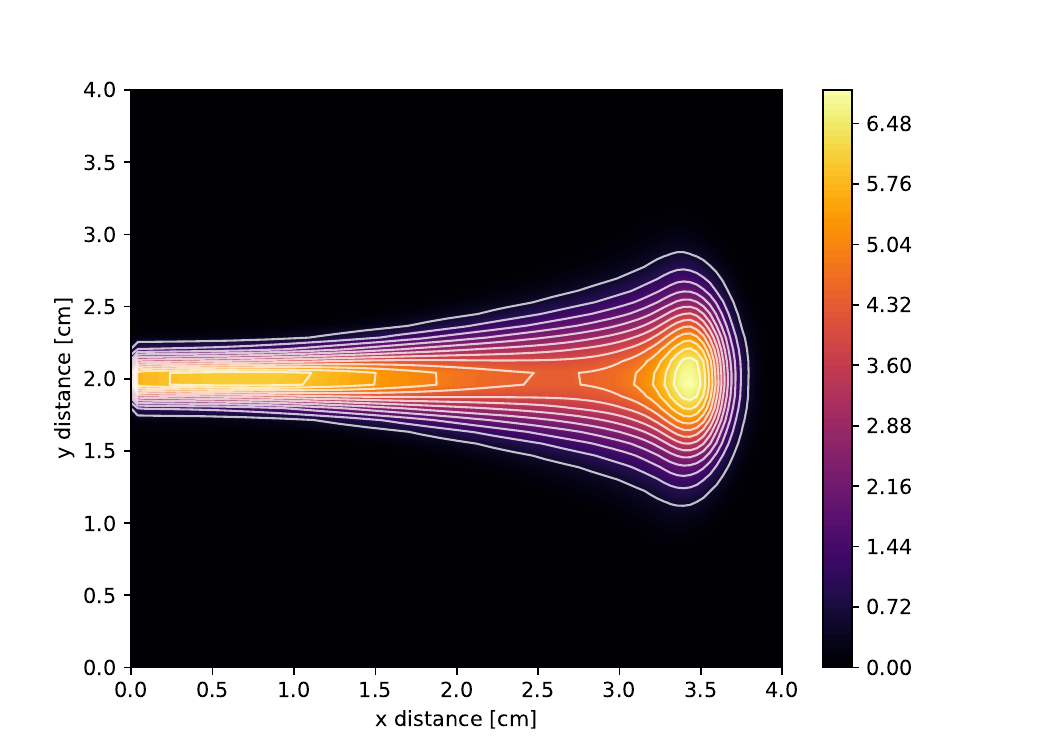}
      \caption{Reference dose.}
    \end{subfigure}
    \caption{Example \ref{ex:full-model-sensitivities}. Sensitivity
      with respect to $\alpha$. In the presence of energy straggling,
      the sensitivity profile is more spread out than in
      \autoref{fig:dose-sens-alpha-no-straggling}. The finite
      difference method still underestimates the gradient
      magnitude.}\label{fig:dose-sens-alpha}
\end{figure}

\begin{figure}[h!]
  \captionsetup[subfigure]{justification=centering}
  \centering
  \begin{subfigure}[b]{0.32\textwidth}
    \includegraphics[width=\textwidth]{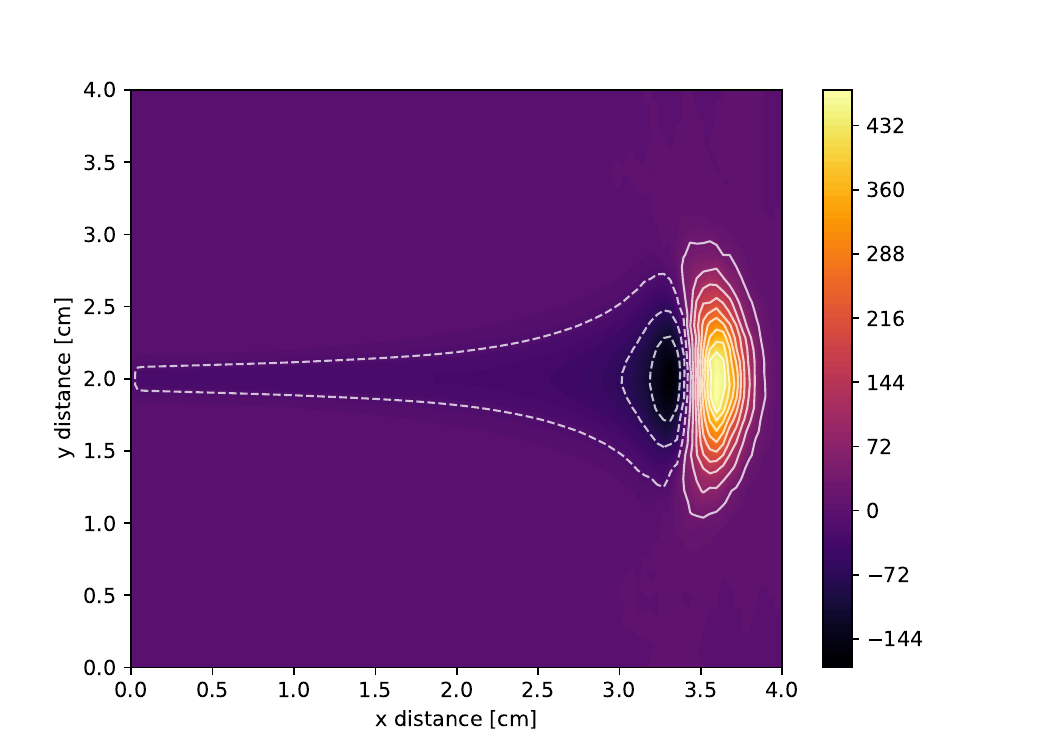}
    \caption{Pathwise sensitivity.}\label{fig:p-pathwise-full}
  \end{subfigure}
  \begin{subfigure}[b]{0.32\textwidth}
    \includegraphics[width=\textwidth]{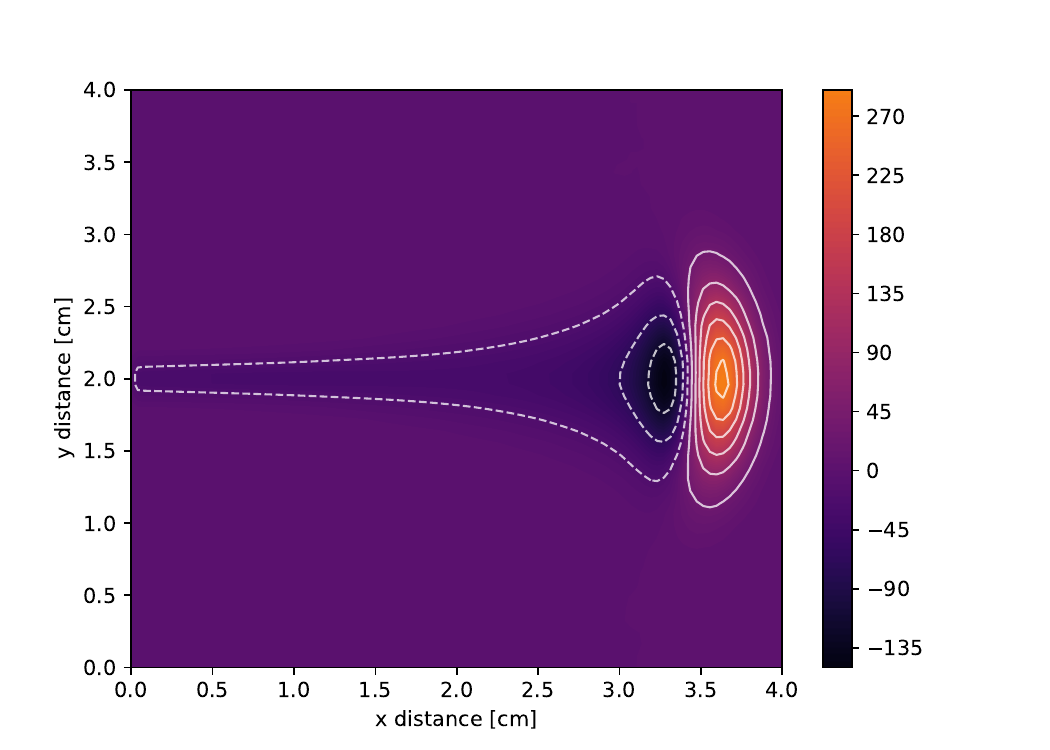}
    \caption{Finite difference sensitivity.}\label{fig:p-fd-full}
  \end{subfigure}
  \begin{subfigure}[b]{0.32\textwidth}
    \includegraphics[width=\textwidth]{plots/dose_sens/straggling/dose_ref.pdf}
    \caption{Reference dose.}
  \end{subfigure}
  \caption{Example \ref{ex:full-model-sensitivities}. Sensitivity with
    respect to $p$. As with $\alpha$, the finite difference estimator
    smooths and underestimates the gradient
    magnitude.}\label{fig:dose-sens-p}
\end{figure}

\begin{figure}[h!]
  \captionsetup[subfigure]{justification=centering}
  \centering
  \begin{subfigure}[b]{0.32\textwidth}
    \includegraphics[width=\textwidth]{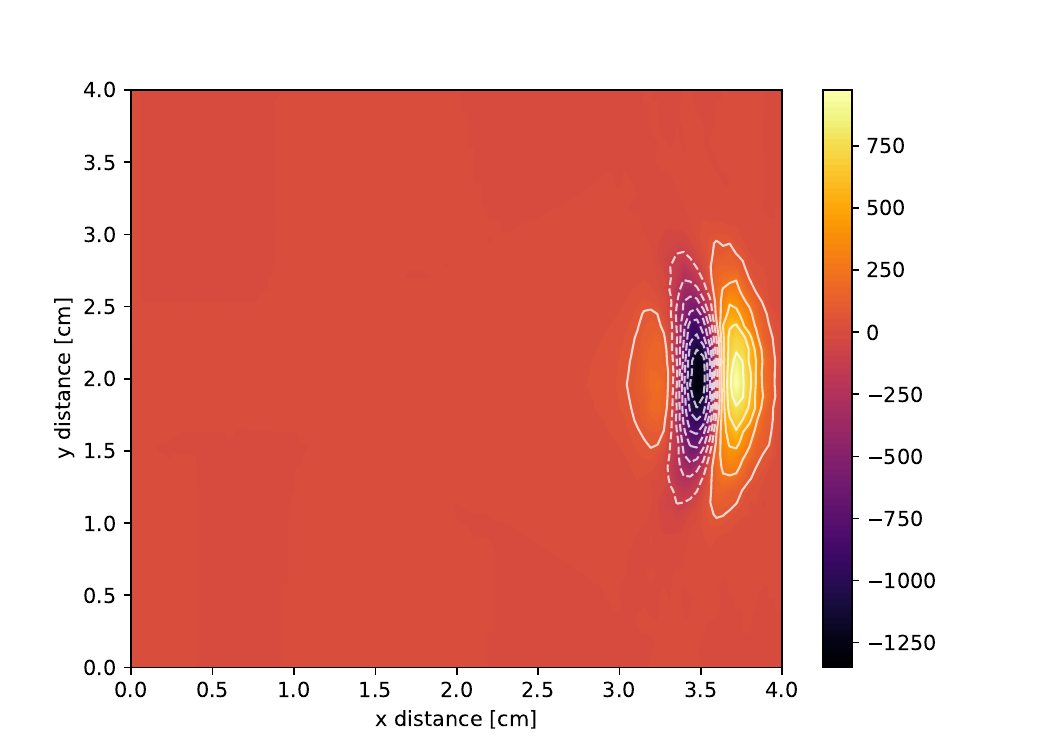}
    \caption{Pathwise sensitivity.}\label{fig:kappa-pathwise}
  \end{subfigure}
  \begin{subfigure}[b]{0.32\textwidth}
    \includegraphics[width=\textwidth]{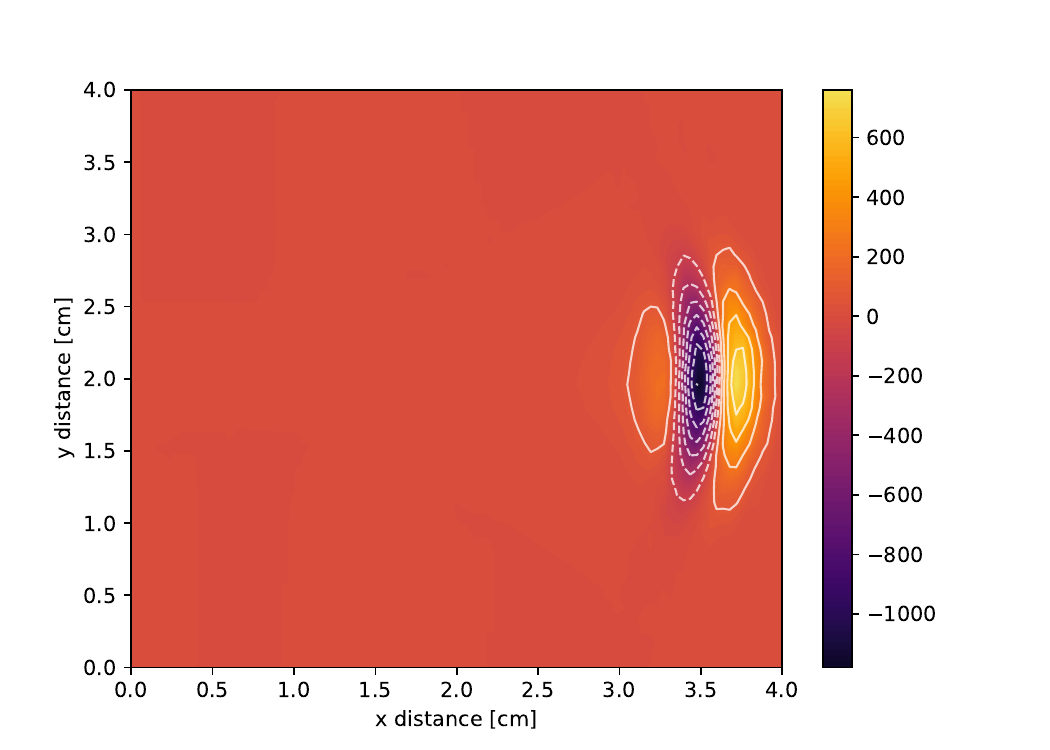}
    \caption{Finite difference sensitivity.}\label{fig:kappa-fd}
  \end{subfigure}
  \begin{subfigure}[b]{0.32\textwidth}
    \includegraphics[width=\textwidth]{plots/dose_sens/straggling/dose_ref.pdf}
    \caption{Reference dose.}
  \end{subfigure}
  \caption{Example \ref{ex:full-model-sensitivities}. Sensitivity with
    respect to $\kappa$. The pathwise method captures sharp features,
    while the finite difference method underestimates the gradient
    magnitude. The profile reflects the widening and lowering of the
    peak caused by straggling.}\label{fig:dose-sens-kappa}
\end{figure}
\end{example}

\begin{example}[Illustration of noise--bias tradeoff in the finite difference approach]\label{ex:noise-bias-fd}
  In Examples \ref{ex:no-straggling-model-sensitivities} and
  \ref{ex:full-model-sensitivities}, we saw evidence that the finite
  difference estimator underestimates dose sensitivities compared to
  the pathwise method. A natural idea to reduce bias is to decrease
  the parameter stepsize $\Delta \theta$. However, doing so increases
  the variance of the estimator \autoref{eq:fd-estimator}, amplifying
  Monte Carlo error.

  We illustrate this bias-variance tradeoff for the parameter
  $\alpha$. Model parameters are fixed as $\alpha=0.022$,
  $\kappa=0.001$, $\epsilon_0=0.005$, with initial conditions and 
  numerical parameters as in \autoref{sec:dose-sensitivity}.

  \autoref{fig:noise-bias-fd} shows sensitivity estimates with $\Delta
  \alpha$ varying over one order of magnitude. With $\Delta \alpha =
  0.1\alpha$ (\autoref{fig:dalpha-large}), the estimator has small
  variance but substantial bias, as seen from the underestimation in
  magnitude. With $\Delta \alpha = 0.0055\alpha$
  (\autoref{fig:dalpha-med}), the bias is smaller and the estimate
  closer to the pathwise benchmark. With $\Delta \alpha = 0.0001\alpha$
  (\autoref{fig:dalpha-small}), statistical noise dominates.

\begin{figure}[h!]
  \captionsetup[subfigure]{justification=centering}
  \centering
  \begin{subfigure}[b]{0.32\textwidth}
    \includegraphics[width=\textwidth]{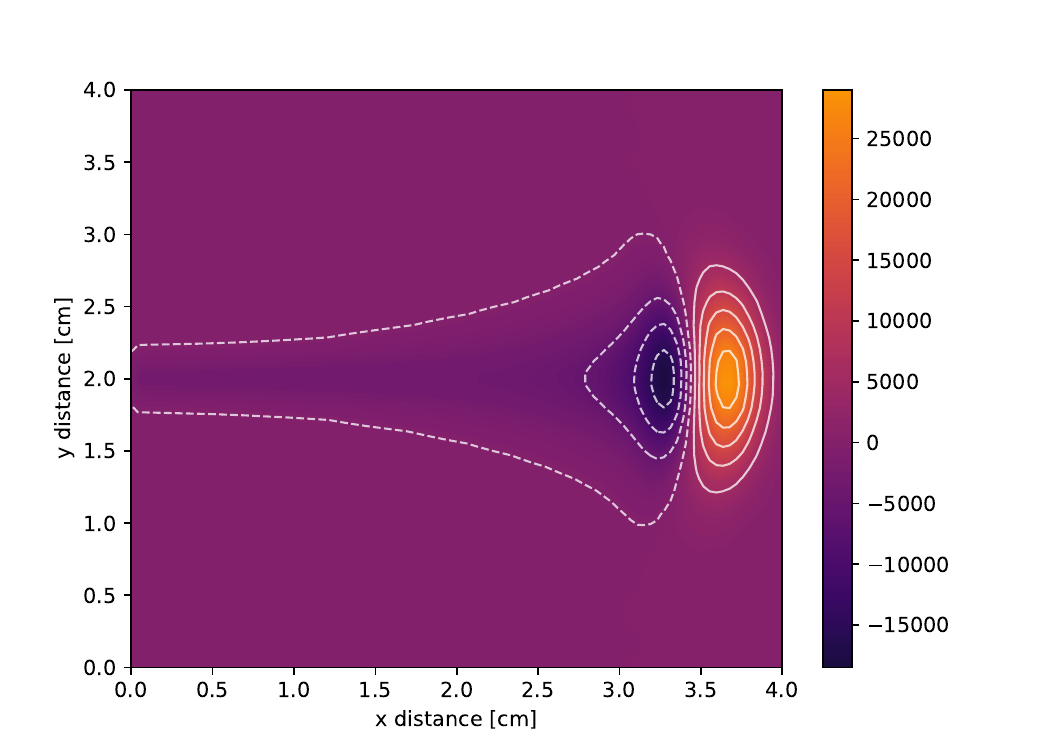}
    \caption{$\Delta \alpha = 0.1 \alpha$}\label{fig:dalpha-large}
  \end{subfigure}
  \begin{subfigure}[b]{0.32\textwidth}
    \includegraphics[width=\textwidth]{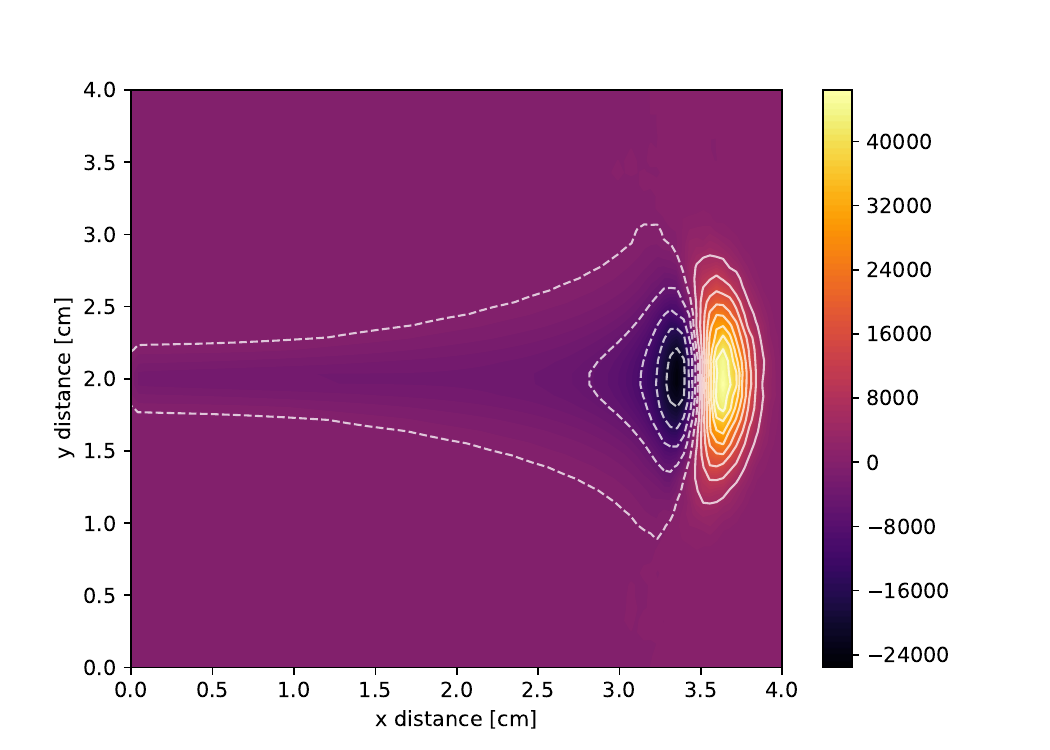}
    \caption{$\Delta \alpha = 0.0055 \alpha$}\label{fig:dalpha-med}
  \end{subfigure}
  \begin{subfigure}[b]{0.32\textwidth}
    \includegraphics[width=\textwidth]{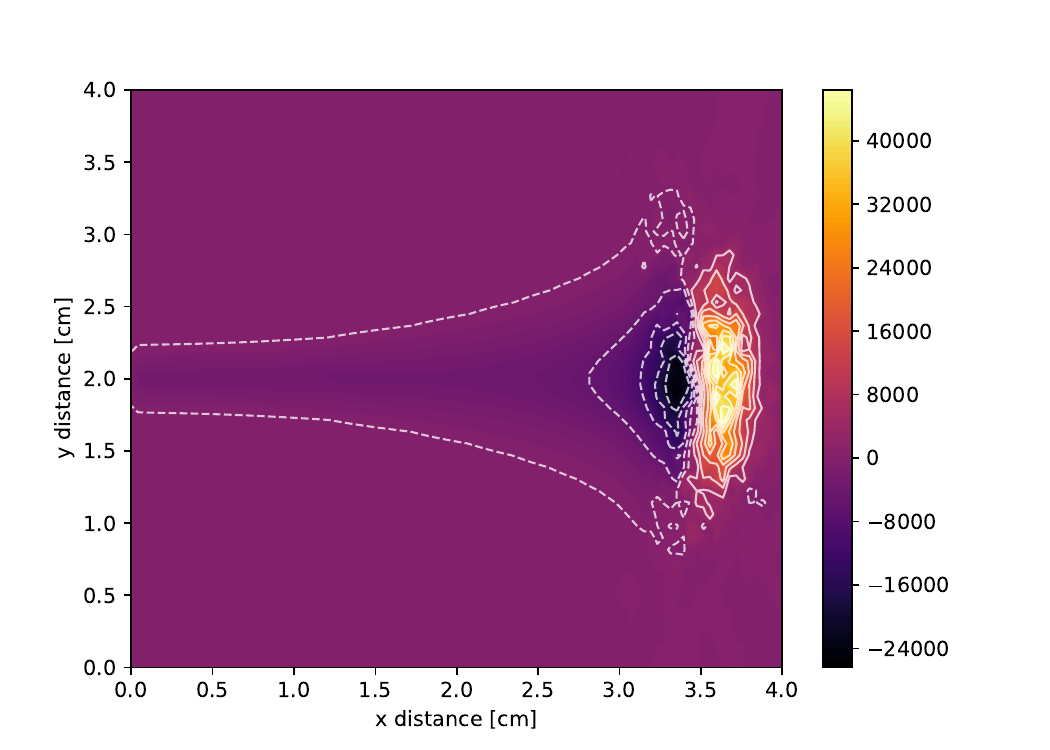}
    \caption{$\Delta \alpha = 0.0001 \alpha$}\label{fig:dalpha-small}
  \end{subfigure}
  \caption{Example \ref{ex:noise-bias-fd}: Bias--variance tradeoff in
    finite difference sensitivities. Larger stepsizes yield low
    variance but large bias. Smaller stepsizes reduce bias but
    increase variance.}
  \label{fig:noise-bias-fd}
\end{figure}

\begin{remark}[Bias-variance tradeoff]
  The finite difference estimator involves an intrinsic
  tradeoff. Decreasing the stepsize $\Delta \theta$ reduces bias but
  increases variance, requiring more Monte Carlo samples to control
  noise. For the central difference estimator, the optimal balance
  between $\Delta \theta$ and sample size $N$ yields a slow asymptotic
  rate of $\mathcal{O}(N^{-2/5})$ (see table~7.1 of
  \cite{glasserman2004monte}, and also
  \cite{glynn1989optimization,fox1989replication}). This is
  significantly worse than the $\mathcal{O}(N^{-1/2})$ variance
  reduction achievable with the pathwise method.
\end{remark}
\end{example}

\section{Conclusion}
\label{sec:conclusion}

We have developed a structure-preserving framework for stochastic
proton transport with continuous energy loss, range straggling and
angular diffusion. The proposed discretisations preserve the key
constraints of the model by maintaining positivity of the energy
variable and the spherical geometry of the angular component, while
remaining compatible with standard strong approximations for SDEs
through a logarithmic transformation for energy and exponential-map
updates for direction.

For the state dynamics, the framework admits a Milstein extension with
strong order-$1$ convergence in the mean-square sense. We also derived
coupled pathwise sensitivity equations with respect to the
Bragg--Kleeman parameters and the straggling coefficient, and showed
how these sensitivities may be combined with regularised observables
to obtain gradient estimators for dose-related quantities. The
numerical experiments demonstrate that the resulting pathwise
estimators capture sharper sensitivity profiles and exhibit more
favourable empirical behaviour than finite-difference benchmarks.

Future work will focus on extending the framework to include discrete
inelastic events, heterogeneous media and more general classes of
observables. It would also be of interest to develop weak-approximation
and variance-reduction methodologies for expectation-level quantities,
and to integrate the present pathwise framework into optimisation,
uncertainty quantification and Bayesian inference pipelines.

\section*{Acknowledgments}

The Python code that reproduces all numerical experiments and plots in
this paper can be found at \url{10.5281/zenodo.17053272}.

VC is supported by a scholarship from the Statistical Applied
Mathematics at Bath (SAMBa) EPSRC Centre for Doctoral Training (CDT)
at the University of Bath under the project EP/S022945/1. TP is
supported by the EPSRC programme grant Mathematics of Radiation
Transport (MaThRad) EP/W026899/2 and the Leverhulme Trust
RPG-2021-238. All of this support is gratefully acknowledged.

\printbibliography

\end{document}